\documentclass[11pt]{article}
\usepackage[margin=1in]{geometry} 
\usepackage{amssymb,amsfonts,amsmath,bbm,mathrsfs,stmaryrd}
\usepackage{xcolor}
\usepackage{url}

\usepackage{enumerate}

\usepackage[colorlinks,
             linkcolor=black!75!red,
             citecolor=blue,
             pdftitle={},
             pdfproducer={pdfLaTeX},
             pdfpagemode=None,
             bookmarksopen=true
             bookmarksnumbered=true]{hyperref}

\usepackage{tikz}
\usetikzlibrary{arrows,calc,decorations.pathreplacing,decorations.markings,intersections,shapes.geometric,through,fit,shapes.symbols,positioning,decorations.pathmorphing}

\usepackage{braket}

\usepackage[amsmath,thmmarks,hyperref]{ntheorem}
\usepackage{cleveref}

\creflabelformat{enumi}{#2(#1)#3}

\crefname{section}{Section}{Sections}
\crefformat{section}{#2Section~#1#3} 
\Crefformat{section}{#2Section~#1#3} 

\crefname{subsection}{\S}{\S\S}
\AtBeginDocument{%
  \crefformat{subsection}{#2\S#1#3}%
  \Crefformat{subsection}{#2\S#1#3}%
}

\theoremstyle{plain}

\newtheorem{lemma}{Lemma}[section]
\newtheorem{proposition}[lemma]{Proposition}
\newtheorem{corollary}[lemma]{Corollary}
\newtheorem{theorem}[lemma]{Theorem}

\theoremstyle{nonumberplain}
\newtheorem{theoremN}{Theorem}
\newtheorem{propositionN}{Proposition}

\theoremstyle{plain}
\theorembodyfont{\upshape}
\theoremsymbol{\ensuremath{\blacklozenge}}

\newtheorem{definition}[lemma]{Definition}
\newtheorem{example}[lemma]{Example}
\newtheorem{remark}[lemma]{Remark}

\crefname{definition}{definition}{definitions}
\crefformat{definition}{#2definition~#1#3} 
\Crefformat{definition}{#2Definition~#1#3} 

\crefname{ex}{example}{examples}
\crefformat{example}{#2example~#1#3} 
\Crefformat{example}{#2Example~#1#3} 

\crefname{remark}{remark}{remarks}
\crefformat{remark}{#2remark~#1#3} 
\Crefformat{remark}{#2Remark~#1#3} 

\crefname{convention}{convention}{conventions}
\crefformat{convention}{#2convention~#1#3} 
\Crefformat{convention}{#2Convention~#1#3} 

\crefname{notation}{notation}{notations}
\crefformat{notation}{#2notation~#1#3} 
\Crefformat{notation}{#2Notation~#1#3} 

\crefname{question}{question}{questions}
\crefformat{question}{#2question~#1#3} 
\Crefformat{question}{#2Question~#1#3}

\crefname{lemma}{lemma}{lemmas}
\crefformat{lemma}{#2lemma~#1#3} 
\Crefformat{lemma}{#2Lemma~#1#3} 

\crefname{proposition}{proposition}{propositions}
\crefformat{proposition}{#2proposition~#1#3} 
\Crefformat{proposition}{#2Proposition~#1#3} 

\crefname{corollary}{corollary}{corollaries}
\crefformat{corollary}{#2corollary~#1#3} 
\Crefformat{corollary}{#2Corollary~#1#3} 

\crefname{theorem}{theorem}{theorems}
\crefformat{theorem}{#2theorem~#1#3} 
\Crefformat{theorem}{#2Theorem~#1#3} 

\crefname{enumi}{}{}
\crefformat{enumi}{(#2#1#3)}
\Crefformat{enumi}{(#2#1#3)}

\crefname{assumption}{assumption}{Assumptions}
\crefformat{assumption}{#2assumption~#1#3} 
\Crefformat{assumption}{#2Assumption~#1#3} 

\crefname{equation}{}{}
\crefformat{equation}{(#2#1#3)} 
\Crefformat{equation}{(#2#1#3)}

\theoremstyle{nonumberplain}
\theoremsymbol{\ensuremath{\blacksquare}}

\newtheorem{proof}{Proof}

\newcommand\bC{{\mathbb C}}

\newcommand\bF{{\mathbb F}}

\newcommand\bK{{\mathbb K}}

\newcommand\bN{{\mathbb N}}

\newcommand\bZ{{\mathbb Z}}

\newcommand\cP{{\mathcal P}}

\numberwithin{equation}{section}
\renewcommand{\theequation}{\thesection-\arabic{equation}}

\newcommand\numberthis{\addtocounter{equation}{1}\tag{\theequation}}

\newcommand{\qedhere}{\mbox{}\hfill\ensuremath{\blacksquare}}

\title{Recursive sequences attached to modular representations of finite groups}
\author{Alexandru Chirvasitu, Tara Hudson and Aparna Upadhyay}

\begin{document}

\date{}

\newcommand{\Addresses}{{
  \bigskip
  \footnotesize

  \textsc{Department of Mathematics, University at Buffalo, Buffalo,
    NY 14260-2900, USA}\par\nopagebreak \textit{E-mail address}:
  \texttt{achirvas@buffalo.edu}

  \medskip

  \textsc{Department of Mathematics, University at Buffalo, Buffalo,
    NY 14260-2900, USA}\par\nopagebreak \textit{E-mail address}:
  \texttt{tarahuds@buffalo.edu}

  \medskip
  
  \textsc{Department of Mathematics, University at Buffalo, Buffalo,
    NY 14260-2900, USA}\par\nopagebreak \textit{E-mail address}:
  \texttt{aparnaup@buffalo.edu}

}}

\maketitle

\begin{abstract}
  The core of a finite-dimensional modular representation $M$ of a finite group $G$ is its largest non-projective summand. We prove that the dimensions of the cores of $M^{\otimes n}$ have algebraic Hilbert series when $M$ is Omega-algebraic, in the sense that the non-projective summands of $M^{\otimes n}$ fall into finitely many orbits under the action of the syzygy operator $\Omega$. Similarly, we prove that these dimension sequences are eventually linearly recursive when $M$ is what we term $\Omega^{+}$-algebraic. This partially answers a conjecture by Benson and Symonds. Along the way, we also prove a number of auxiliary permanence results for linear recurrence under operations on multi-variable sequences.
\end{abstract}


\vspace{.5cm}



\section*{Introduction}

Let $G$ be a finite group, $k$ a field whose characteristic $p$ divides $|G|$, and $\mathrm{mod}~kG$ the category of $G$-modules, finite-dimensional over $k$. The paper \cite{bs} studies the asymptotic behavior as $n\to\infty$ of the {\it cores} of the tensor powers $M^{\otimes n}$ for $M\in\mathrm{mod}~kG$, where by definition
\begin{equation*}
  core(M) = core_G(M) :=\text{ the largest non-projective summand of $M$}.
\end{equation*}
The initial motivation for the present paper was \cite[Conjecture 13.3]{bs}, stating that the dimensions
\begin{equation}\label{eq:cns}
  c_n^G(M):=\dim core\left(M^{\otimes n}\right)
\end{equation}
form an eventually linearly recursive sequence. A likely more tractable version is \cite[Conjecture 14.2]{bs}, which restricts the class of $G$-modules under consideration. To make sense of that statement, recall (e.g. \cite[\S 1.5]{bnsn-rep1} or \cite[discussion following Lemma 2.7]{bs}) that for a finite-dimensional $G$-module $M$ one writes
\begin{itemize}
\item $\Omega M$ for the kernel of a {\it projective cover} $P\to M$;
\item $\Omega^{-1} M$ for the cokernel of an {\it injective hull} $M\to I$.
\end{itemize}
These are not quite endofunctors on the category of modules, because projective/injective covers are not functorial, but they do descend to endofunctors of the {\it stable module category}
\begin{equation*}
  \mathrm{stmod}~kG:=\mathrm{mod}~kG/\mathrm{proj},
\end{equation*}
defined as having the same objects as the category $\mathrm{mod}~kG$ of finite-dimensional $G$-modules and whose morphisms are obtained by annihilating those module morphisms that factor through projective (or equivalently, injective) objects; see e.g. \cite[Chapter I]{hap}.

In $\mathrm{stmod}~kG$ $\Omega$ and $\Omega^{-1}$ are indeed (as the notation suggests) mutually inverse functors:
\begin{equation*}
  \Omega\left(\Omega^{-1} M\right)\cong core(M)\cong \Omega^{-1}\left(\Omega M\right)
\end{equation*}
already holds in $\mathrm{mod}~kG$, and stabilization has the effect of identifying $M$ and its core. Given that
\begin{itemize}
\item we often ignore projective summands, as the problems under consideration require;
\item and $\Omega^{\pm 1}$ are endofunctors of $\mathrm{stmod}~kG$,
\end{itemize}
we will often treat them as functors, referring to them as such, composing them, etc. With this in place, recall \cite[Definition 14.1]{bs}:
\begin{definition}\label{def:omega}
  A $G$-module is {\it Omega-algebraic} (or $\Omega$-algebraic) if the non-projective indecomposable summands of the various tensor powers $M^{\otimes n}$ fall into finitely many orbits under the action of $\bZ$ via $\Omega$.
\end{definition}
This means that the functor $M\otimes-$ can be recast as a matrix $T$ with entries in the Laurent polynomial ring $\bZ[\Omega^{\pm 1}]$. We can restrict this further (see \Cref{se.omega}) for a fuller discussion:
\begin{definition}\label{def:omega+}
  $M\in \mathrm{mod}~kG$ is {\it Omega$^+$(or $\Omega^+$)-algebraic} if
  \begin{itemize}
  \item it is $\Omega$-algebraic in the sense of \Cref{def:omega}, and
  \item the representatives
    \begin{equation*}
      N_1=k,\ N_2,\ \cdots
    \end{equation*}
    for the $\Omega$-orbits of the simple subquotients of $M^{\otimes n}$, $n\in \bN$ can be chosen so that the entries of the matrix $T$ given by $M\otimes -$ are polynomials in $\bN[\Omega]$ (rather than {\it Laurent} polynomials).
  \end{itemize}

  We define {\it Omega$^-$-algebraic} modules similarly, substituting $\bN[\Omega^{-1}]$ for $\bN[\Omega]$ above.
\end{definition}

Our main results pertaining to these classes of modules are as follows. First, regarding \cite[Conjecture 14.2]{bs}, we have (\Cref{th:cj142-weak} and \Cref{cor:sufflrg})

\begin{theoremN}
  Let $M\in \mathrm{mod}~kG$. The sequence \Cref{eq:cns} is eventually linearly recursive if $M$ is either $\Omega^+$ or $\Omega^-$-algebraic.

  Consequently, the same holds if $M$ is of the form $\Omega^d N$ for $\Omega$-algebraic $N$ and sufficiently large (or sufficiently small) $d\in \bZ$.
  \qedhere
\end{theoremN}

A sequence ${\bf a}=(a_n)$ is eventually linearly recursive precisely when its {\it Hilbert series}
\begin{equation*}
  H_{\bf a}(t) = \sum_n a_nt^n
\end{equation*}
is rational (see \Cref{se:rec} below for a lengthier discussion of linear recursion). This condition can be weakened in various ways, e.g. by requiring that $H_{\bf a}$ be only {\it algebraic} (i.e. that it satisfy a polynomial equation with coefficients in the field of rational functions in $t$). To return to $G$-modules, for $\Omega$- (rather than $\Omega^{\pm}$-)algebraic modules we have \Cref{th:wkagain}:

\begin{theoremN}
  For an $\Omega$-algebraic $M\in \mathrm{mod}~kG$ the sequence \Cref{eq:cns} has algebraic Hilbert series. \qedhere
\end{theoremN}

This will require a bit of a detour, as we need various results to the effect that recursion and related properties (e.g. having an algebraic Hilbert series) are invariant under various constructions involving sequences or, more generally, {\it multi-sequences} (\Cref{subse:mul}). Such results are presumably of some independent interest, and they appear throughout \Cref{se:rec,se:conv}. A small sampling (\Cref{def:op} and \Cref{pr:qpn}):

\begin{propositionN}
  Consider
  \begin{itemize}
  \item an eventually-linearly-recursive sequence $(P_n)_n$ of polynomials in $x$ over a field $\bK$;
  \item an eventually-linearly-recursive sequence ${\bf a}=(a_n)_n$ in $\bK$,
  \end{itemize}
  and denote by
  \begin{equation*}
    P\triangleright{\bf a} = \sum_k c_k a_k
  \end{equation*}
  the {\bf convolution} of a polynomial $P(x)=\sum c_k x^k$ with ${\bf a}$.

  Then, the sequence $(P_n\triangleright {\bf a})_n$ is eventually linearly recursive. \qedhere
\end{propositionN}
Such convolution operations feature prominently in the proofs of the above-mentioned theorems, and they form the focus of \Cref{se:conv} and part of \Cref{se:rec}.

In \Cref{se.poly} we prove that various generalizations of $c_n^G(M)$ are eventually linearly recursive or algebraic, broadening the scope of the discussion. Specifically, an aggregate of \Cref{th:po} and \Cref{th:lrnt} reads

\begin{theoremN}
  Let $M\in\mathrm{mod}~kG$ and $F$ a functor from $\mathrm{mod}~kG$ to finite-dimensional vector spaces that is either exact or of the form $\mathrm{Hom}_G(S,-)$ for a simple $G$-module $S$.
  \begin{enumerate}[(a)]
  \item If $(P_n)_n$ is an eventually linearly recursive sequence of polynomials in $\bN[x]$ then the sequences
    \begin{equation*}
      n\mapsto \dim F\left(P_n\Omega M\right)\quad\text{or}\quad \dim F\left(P_n\Omega^{-1} M\right)
    \end{equation*}
    are eventually linearly recursive. 
  \item On the other hand, if $P_n$ are {\bf Laurent} polynomials, the same sequences have algebraic Hilbert series. \qedhere
  \end{enumerate}
\end{theoremN}

Finally, \Cref{se.examples} contains examples of sequences $c_n^G(M)$ and analogues for specific modules/groups, illustrating the main results outlined above.

\subsection*{Some notation}

We write $\bN$ for $\bZ_{\ge 0}$. Throughout,
\begin{itemize}
\item $G$ is a finite group;
\item $k$ is a field of positive characteristic $p$ (typically dividing $|G|$; otherwise most of the discussion below will be trivial);
\item $\mathrm{Vect}$ (respectively $\mathrm{Vect}^f$) means (finite-dimensional) $k$-vector spaces,
\item and as in the Introduction, $\mathrm{mod}~kG$ denotes the category of $k$-finite-dimensional $G$-modules.
\end{itemize}
We write $\ell(M)$ for the length of a module $M$, so $\ell=\dim$ for plain vector spaces.

Recall the quantities $c_n^G(M)$ from the Introduction (\Cref{eq:cns}). Prompted by \cite[Remark 2.5(i)]{bs} on the resilience of the invariant $\gamma_G(M)$ to replacing $c_n^G(M)$ with the length or the length of the socle of $core\left(M^{\otimes n}\right)$ we write
\begin{itemize}
\item $d_n^G(M)$ for the length of the socle of $core\left(M^{\otimes n}\right)$;
\item $l_n^G(M)$ for the length of $core\left(M^{\otimes n}\right)$;
\item $s_n^G(M)$ for the number of indecomposable summands of $core\left(M^{\otimes n}\right)$.
\end{itemize}

\subsection*{Acknowledgements}

We are grateful for numerous highly instructive exchanges with David Hemmer.

AC acknowledges support through NSF grant DMS-2001128.

\section{Generalities on recursion}\label{se:rec}

\subsection{Multi-sequences}\label{subse:mul}

\cite[Chapter 4]{enum1} is a good reference for the material on linear recursive sequences needed below. Since we are interested in sequences (and sometimes multi-sequences, i.e. $a_{m,n,\cdots}$) of either complex numbers or polynomials, it will be convenient to keep in mind that most of the discussion below makes sense over commutative rings.

\begin{definition}\label{def:shifts}
  Let $R$ be a commutative ring and $r$ a positive integer.

  The elements of the product space $X=R^{\bN^r}$ are {\it $R$-valued} {\it $r$-sequences}, or {\it $r$-dimensional (multi-) sequences}. When we do not specify $r$ we use the phrase {\it multi-sequence}.
  
  For $1\le i\le r$, the {\it $i^{th}$ shift $S_i$} on $X$ is the operator that shifts the $i^{th}$ index of a multi-sequence up (and hence shifts the multi-sequence ``leftward'' along the $i^{th}$ direction).

  When $r=1$ (i.e. we work with plain sequences) we will often write $S$ for the only shift operator $S_1$.

  For a tuple ${\bf n}=(n_1,\cdots,n_r)$ of non-negative integers, we write $S^{\bf n}$ for the product
  \begin{equation*}
    S_1^{n_1}\cdots S_r^{n^r}.
  \end{equation*}
\end{definition}

To illustrate:

\begin{example}
  For $r=1$, given a sequence ${\bf a}=(a_n)_n$, its shift $S{\bf a}$ is $(a_{n+1})_{n}$.

  On the other hand, for $r=2$ and ${\bf a}=(a_{m,n})_{m,n}$ we have
  \begin{equation*}
    S_2{\bf a} = (a_{m,n+1})_{m,n}. 
  \end{equation*}
\end{example}

The fundamental result, to be used extensively below, is the characterization of (eventually) linear recursive sequences given in \cite[Theorem 4.1.1]{enum1}. Paraphrasing that result slightly, extending it to algebraically-closed fields more general than $\bC$ (e.g. \cite[Theorems 4.1 and 4.3]{kp}), and supplementing it with a shift criterion, we have 

\begin{theorem}\label{th:st-rec}
  Let ${\bf a}=(a_n)_n$ be a sequence valued in a field $\bK$. The following conditions are equivalent.
  \begin{enumerate}[(a)]
  \item\label{item:1} The Hilbert series
    \begin{equation*}
      H_{\bf a}(t):=\sum_{n}a_n t^n
    \end{equation*}
    attached to the sequence is a rational function in $t$. 
  \item\label{item:2} The sequence is eventually linear recursive, in the sense that
    \begin{equation*}
      a_{n+T}+c_1a_{n+T-1}+\cdots+c_T a_n=0
    \end{equation*}
    for some complex numbers $c_i$ and sufficiently large $n$. 
  \item\label{item:3} There are polynomials $p_s$ and elements $\gamma_s$ of the algebraic closure $\overline{\bK}\supseteq \bK$ such that
    \begin{equation*}
      a_n = \sum_s p_s(n)\gamma_s^n
    \end{equation*}
    for sufficiently large $n$.
  \item\label{item:4} The vector subspace of $\bK^{\bN}$ spanned by the shifts $S^d{\bf a}$, $d\in \bN$ is finite-dimensional. \qedhere
  \end{enumerate}  
\end{theorem}

We will soon see that for multi-sequences things are more complicated. For instance, the rationality of the Hilbert series (\Cref{item:1} of \Cref{th:st-rec}) and the finite-dimensionality of the space of shifts (condition \Cref{item:4}) part ways.

For a start, we identify a particularly well-behaved class of multi-sequences: the {\it multi-C-finite} ones of \cite[\S 2.2.2]{z-hol}. 

\begin{theorem}\label{th:multirat}
  Let ${\bf a}=(a_{\bf n})_{{\bf n}=(n_1,\cdots,n_r)}$ be a multi-sequence valued in a field $\bK$. The following conditions are equivalent.
  \begin{enumerate}[(a)]
  \item\label{item:5} The Hilbert series
    \begin{equation*}
      H_{\bf a}(t_1,\cdots,t_r):=\sum_{\bf n}a_{\bf n} t_1^{n_1}\cdots t_r^{n_r}
    \end{equation*}
    is a function of the form
    \begin{equation*}
      \frac{P(t_1,\cdots,t_r)}{Q_1(t_1)\cdots Q_r(t_r)}
    \end{equation*}
    for an $r$-variable polynomial $P$ and single-variable polynomials $Q_i$.
  \item\label{item:10} There are $r$-variable polynomials $p_s$ and tuples
    \begin{equation*}
      {\bf \gamma}_s = (\gamma_{s,1},\cdots,\gamma_{s,r})\in \overline{\bK}^r
    \end{equation*}
    such that
    \begin{align*}
      a_{\bf n} = a_{n_1,\cdots,n_r} &= \sum_s p_s(n_1,\cdots,n_r){\bf \gamma}_s^{\bf n}\\
                                     &=\sum_s p_s(n_1,\cdots,n_r)\gamma_{s,1}^{n_1}\cdots\gamma_{s,r}^{n_r}\numberthis\label{eq:gammas}
    \end{align*}
    for all but finitely many tuples ${\bf n}=(n_1,\cdots,n_r)$.
  \item\label{item:6} The vector subspace of $\bK^{\bN^r}$ spanned by the shifts $S^{\bf n}{\bf a}$, ${\bf n}\in \bN^r$ is finite-dimensional.
  \item\label{item:11} For each $1\le i\le r$, the vector subspace of $\bK^{\bN^r}$ spanned by the shifts $S_i^{n}{\bf a}$, ${n}\in \bN$ is finite-dimensional.
  \end{enumerate}  
\end{theorem}
\begin{proof}
  {\bf \Cref{item:5} $\Rightarrow$ \Cref{item:10}.} By \cite[Theorem 1]{ges} we may as well assume that the polynomials $Q_i$ have non-vanishing free terms, and hence we can factor them as
  \begin{equation*}
    Q_i(t_i) = (1-\mu_{1,i}t_i)^{m_{1,i}}\cdots (1-\mu_{k_i,i}t_i)^{m_{k_i,i}}
  \end{equation*}
  for distinct (possibly vanishing) elements $\mu_{\bullet,i}$ in the algebraic closure $\overline{\bK}$. A simple computation now shows that we can take the tuples ${\bf \gamma}_s$ in \Cref{item:10} to be 
  \begin{equation*}
    (\gamma_{s,1},\cdots,\gamma_{s,r}) = (\mu_{\bullet,1},\cdots,\mu_{\bullet,r})
  \end{equation*}
  for various choices of `$\bullet$'. 

  {\bf \Cref{item:10} $\Rightarrow$ \Cref{item:11}.} The multi-sequences satisfying either of the two conditions form a linear space, so it is enough to consider a ``monomial'' multi-sequence, of the form
  \begin{equation*}
    a_{n_1,\cdots,n_r} = n_1^{k_1}\cdots n_r^{k_r} \gamma_1^{n_1}\cdots \gamma_r^{n_r}. 
  \end{equation*}
  Without loss of generality, it is enough to show that $S_1^n {\bf a}$, $n\in \bN$ span a finite-dimensional space. Rescaling $S_1$ by $\gamma_1$, the exponential part $\gamma_1^{n_1}\cdots \gamma_r^{n_r}$ can be dropped entirely, as can all factors independent of $n_1$. In short, we can consider 
  \begin{equation*}
    a_{n_1,\cdots,n_r} = n_1^{k_1}
  \end{equation*}
  instead. We are now back in the plain-sequence case, where we can fall back on \Cref{th:st-rec}. 
  
  {\bf \Cref{item:6} $\Leftrightarrow$ \Cref{item:11}.} The rightward implication is obvious, whereas its converse follows from the fact that the shifts $S_i$ (for $1\le i\le r$) commute: by \Cref{item:11}, for each $i$ there is some $N_i$ such that every $S_i^d {\bf a}$ is a linear combination of the $S_i^n {\bf a}$ for $n<N_i$. But then, by the noted commutation,
  \begin{equation*}
    S_1^{d_1}\cdots S_r^{d_r} {\bf a} = \text{a linear combination of }S_1^{n_1}\cdots S_r^{n_r} {\bf a},\quad n_i<N_i.
  \end{equation*}
  
  {\bf \Cref{item:11} $\Rightarrow$ \Cref{item:5}.} Condition \Cref{item:11} says that for each $1\le i\le r$ there is some polynomial $Q_i(t_i)$ such that the exponents of $t_i$ in $Q_i(t_i)H_{\bf a}(t_1,\cdots,t_r)$ are uniformly bounded. Applying this to all $i$, the product
  \begin{equation*}
    Q_1(t_1)\cdots Q_r(t_r)H_{\bf a}(t_1,\cdots,t_r)
  \end{equation*}
  has only finitely many monomials; in other words, it is a polynomial.
\end{proof}

\begin{definition}\label{def:evrec}
  A multi-sequence over a field $\bK$ is
  \begin{itemize}
  \item {\it C-finite} if it satisfies the equivalent conditions of \Cref{th:multirat}.
  \item {\it rational} if its Hilbert series is rational.
  \item {\it algebraic} if its Hilbert series $H(t)$ is algebraic, in the sense that it satisfies an equation
    \begin{equation*}
      P_d(t) H(t)^d+\cdots+P_1(t)H(t)+P_0(t)=0
    \end{equation*}
    in $\bK[[t]]$, where the $P_i$ are polynomials (not all vanishing);
  \end{itemize}
  see \cite[Definition 6.1.1]{enum2}.
  
  C-finiteness and rationality are equivalent in the 1-dimensional case, where we also refer to such (plain, 1-dimensional) sequences as {\it eventually linear(ly) recursive}.
\end{definition}

\begin{remark}
  Our linear (or linearly) recursive sequences are those studied in \cite[\S 4.1]{enum1}, as well as the {\it recurrence sequences} of \cite[\S 1.1.1]{ev-rec}: it is assumed, in particular, that they satisfy recurrence relations of the form
  \begin{equation}\label{eq:rectemplate}
    a_{n+T}=c_{T-1}a_{n+T-1}+\cdots+c_0 a_n,\quad c_0\text{ not a zero divisor}
  \end{equation}
  (in whatever commutative ring the coefficients $c_i$ belong to) for {\it all} $n$. This convention rules out, for instance, sequences that are eventually zero. This is the reason for requiring the modifier `eventually' in \Cref{def:evrec} and for introducing the pithier term `rational' (justified by \Cref{th:multirat}).
\end{remark}

\begin{example}
  The rationality of \Cref{def:evrec} is weaker than C-finiteness. This is clear for instance from condition \Cref{item:5} of \Cref{th:multirat}, which requires that the denominator be separable as a product of univariate polynomials, but can also be seen by exhibiting a 2-dimensional sequence with rational Hilbert series whose shifts span an infinite-dimensional space.

  Take, say,
  \begin{equation*}
    {\bf a} = (a_{m,n})_{m,n},\quad a_{m,n}=
    \begin{cases}
      1 &\text{if }m=n\\
      0 &\text{otherwise}
    \end{cases}
  \end{equation*}
  The shifts $S_2^d {\bf a}$ are linearly independent, but the Hilbert series is the rational 2-variable function $\frac 1{1-xy}$.
\end{example}

\begin{theorem}\label{th:prodisrat}
  Rationality for multi-sequences in the sense of \Cref{def:evrec} enjoys the following permanence properties.
  \begin{enumerate}[(a)]
  \item\label{item:7} Let ${\bf a}$ and ${\bf b}$ be two $r$-sequences over a field $\bK$, with
    \begin{itemize}
    \item ${\bf a}$ C-finite, and
    \item ${\bf b}$ rational or C-finite. 
    \end{itemize}
    then, their product
    \begin{equation*}
      {\bf a}{\bf b}:= (a_{\bf n} b_{\bf n})_{{\bf n}=(n_1,\cdots,n_r)}
    \end{equation*}
    is rational or C-finite respectively.
  \item\label{item:9} If ${\bf a}$ is a rational $r$-sequence over $\bK$ and for each $(r-1)$-tuple $(n_1,\cdots,n_{r-1})$ only finitely many $a_{n_1,\cdots,n_r}$ are non-zero, the truncation
    \begin{equation*}
      {\bf a'}\in \bK^{\bN^{r-1}},\quad a'_{n_1,\cdots,n_{r-1}} = \sum_{n_r}a_{n_1,\cdots,n_{r}}
    \end{equation*}
    is again rational.
  \item\label{item:matpr} If ${\bf a}$ is a rational 2-sequence such that for each $m$ the number of non-zero $a_{m,n}$ is finite and ${\bf b}$ is a rational sequence, then the ``matrix product'' sequence
    \begin{equation*}
      ({\bf a}\bullet {\bf b})_m := \left(\sum_n a_{m,n}b_n\right)_m
    \end{equation*}
    is rational.
  \item\label{item:12} If ${\bf a}$ is a rational (C-finite) $r$-sequence then so is
    \begin{equation*}
      ({\bf b}_{\bf n})_{{\bf n}=(n_1,\cdots,n_r)} := \left(\sum_{i=0}^{n_1}a_{i,n_2\cdots n_r}\right)_{n_1,\cdots,n_r}.
    \end{equation*}
  \item\label{item:8} If ${\bf a}$ is a rational (C-finite) $r$-sequence then so is
    \begin{equation*}
      ({\bf b}_{\bf n})_{{\bf n}=(n_1,\cdots,n_r)} := (a_{dn_1,n_2\cdots n_r})_{n_1,\cdots,n_r},
    \end{equation*}
    for any positive integer $d$.
  \end{enumerate}
\end{theorem}
\begin{proof}
  \Cref{item:7} Since the proofs of the two claims are substantively different, we treat them separately.

  {\bf (Case 1: {\bf a} C-finite, {\bf b} rational)} The argument resembles that in the proof of \cite[Proposition 6.1.11]{enum2}, except it is simpler because we are handling only rational (rather than algebraic) power series.

  The rationality of ${\bf a}\bullet {\bf b}$ will not be affected by altering finitely many terms of ${\bf a}$, so we may as well assume we have an expression \Cref{eq:gammas} for ${\bf a}$. Moreover, by linearity, we can simplify this to
  \begin{equation*}
    a_{n_1,\cdots,n_r} = n_1^{k_1}\cdots n_r^{k_r}\gamma_1^{n_1}\cdots\gamma_r^{n_r}.
  \end{equation*}
  Since the goal is to show that
  \begin{equation*}
    \sum_{n_i}b_{n_1,\cdots,n_r} a_{n_1,\cdots,n_r} t_1^{n_1}\cdots t_r^{n_r} 
  \end{equation*}
  is rational, the change of variables $t_i\mapsto \gamma_i t_i$ further reduces this to 
  \begin{equation*}
    a_{n_1,\cdots,n_r} = n_1^{k_1}\cdots n_r^{k_r},
  \end{equation*}
  and inducting separately on the $k_i$ finally boils down the goal to proving that if
  \begin{equation*}
    H_{\bf b}(t_i) = \sum_{{\bf n}=(n_1,\cdots,n_r)}b_{\bf n} t_1^{n_1}\cdots t_r^{n_r}
  \end{equation*}
  is rational then so is
  \begin{equation*}
    \widetilde{H}_{\bf b}(t_i) := \sum_{{\bf n}=(n_1,\cdots,n_r)}n_1 b_{\bf n} t_1^{n_1}\cdots t_r^{n_r}.
  \end{equation*}
  This is immediate though, because we have
  \begin{equation*}
    \widetilde{H}_{\bf b}(t_i) = t_i \frac {\partial {H}_{\bf b}(t_i)}{\partial t_i}. 
  \end{equation*}
  
  {\bf (Case 2: {\bf a} and {\bf b} C-finite)} This time around it is the proof of \cite[Theorem 4.2, point 2.]{kp} that we adapt.

  Multiplication
  \begin{equation*}
    ({\bf a},{\bf b})\mapsto {\bf a}\bullet {\bf b}
  \end{equation*}
  is bilinear, inducing a linear map $m:A\otimes A\to A$ for $A:=\bK^{\bN^r}$. That map is compatible with shifting (i.e. the shifts act as algebra endomorphisms), so the shifts $S^{\bf n} ({\bf a}\bullet {\bf b})$ are contained in the image of
  \begin{equation*}
    (\text{span of shifts of }{\bf a})\otimes (\text{span of shifts of }{\bf b}) \le A\otimes A
  \end{equation*}
  through the multiplication map $m$. Since both tensorands are finite-dimensional by assumption, so is
  \begin{equation*}
    \mathrm{span}\{S^{\bf n} ({\bf a}\bullet {\bf b})\ |\ {\bf n}\in \bN^r\}. 
  \end{equation*}

  \Cref{item:9} The Hilbert series
  \begin{equation*}
    H_{\bf a'}(t_1,\cdots,t_{r-1})
  \end{equation*}
  of ${\bf a'}$ is obtained from that of ${\bf a}$ by substituting $1$ for the $r^{th}$ variable $t_r$ (we need the vanishing hypothesis for this to make sense). Since we are assuming rationality, we have 
  \begin{equation*}
    H_{\bf a}(t_1,\cdots,t_r) = \frac{A(t_1,\cdots,t_r)}{B(t_1,\cdots,t_r)},\quad A,B\in \bK[t_1,\cdots,t_r];
  \end{equation*}
  and hence will obtain a rational function upon making the substitution $t_r=1$.

  \Cref{item:matpr} The matrix product sequence ${\bf a}\bullet {\bf b}$ is obtained by
  \begin{itemize}
  \item constructing the 2-sequence ${\bf \overline{b}}$ defined by
    \begin{equation*}
      \overline{b}_{m,n} = b_n,
    \end{equation*}
    which is C-finite (for instance because it satisfies condition \Cref{item:11} of \Cref{th:multirat});
  \item then forming the product ${\bf a}{\bf \overline b}$, which is rational by part \Cref{item:7};
  \item and then summing out the second component of the resulting 2-sequence:
    \begin{equation*}
      ({\bf a}\bullet{\bf b})_m = \sum_n \left({\bf a}{\bf \overline{b}}\right)_{m,n};
    \end{equation*}
    this again produces a rational sequence by part \Cref{item:9}. 
  \end{itemize}

  \Cref{item:12} The monomial $b_{n_1\cdots n_r} t_1^{n_1}\cdots t_r^{n_r}$ of the Hilbert series $H_{\bf b}(t_i)$ is, by definition,
  \begin{equation*}
    a_{0,n_2\cdots n_r} t_1^{n_1}\cdots t_r^{n_r} + \cdots + a_{n_1,n_2\cdots n_r} t_1^{n_1}\cdots t_r^{n_r}.
  \end{equation*}
  These are
  \begin{itemize}
  \item the $(0,n_2,\cdots,n_r)$ term of $H_{\bf a}(t_i)$ multiplied by $t_1^{n_1}$;
  \item the $(1,n_2,\cdots,n_r)$ term of $H_{\bf a}(t_i)$ multiplied by $t_1^{n_1-1}$;
  \item $\cdots$
  \item the $(n_1,n_2,\cdots,n_r)$ term of $H_{\bf a}(t_i)$ (multiplied by $1=t_1^{0}$).
  \end{itemize}
  Summing over all tuples ${\bf n}$, this means that $H_{\bf b}$ is obtained from $H_{\bf a}$ by multiplying each term of the latter by $1+t_1+t_1^2+\cdots$. In short:
  \begin{equation*}
    H_{\bf b}(t_1,\cdots,t_r) = \frac 1{1-t_1} H_{\bf a}(t_1,\cdots,t_r). 
  \end{equation*}
  Both versions (rational and C-finite) of the claim follow from this (using part \Cref{item:5} of \Cref{th:multirat} for the C-finite arm of the argument).

  \Cref{item:8} The Hilbert series $H_{\bf b}$ is obtained from the original one $H_{\bf a}$ by
  \begin{itemize}
  \item dropping all monomials $t_1^{n_1}\cdots t_r^{n_r}$ where $n_1$ is {\it not} divisible by $d$,
  \item and then substituting $t_1$ for $t_1^d$ throughout. 
  \end{itemize}
  The first step (dropping monomials) can be achieved by taking the Hadamard product with the C-finite Hilbert series
  \begin{equation*}
    H(t_1,\cdots,t_r) = \frac 1{1-t_1^d},
  \end{equation*}
  so it preserves both rationality and C-finiteness by part \Cref{item:7} of the present result. As for the second step, write
  \begin{equation*}
    H(t_1,\cdots,t_r) = \frac{P(t_1,\cdots,t_r)}{Q(t_1,\cdots,t_r)}
  \end{equation*}
  for coprime polynomials $P$ and $Q$ (this makes sense because polynomial rings are unique factorization domains \cite[\S 9.3, Theorem 7]{df}), with $Q$ either a plain polynomial or a special one, separable as a product $Q_1(t_1)\cdots Q_r(t_r)$ as in part \Cref{item:5} of \Cref{th:multirat}. By \Cref{le:divexp} $P$ and $Q$ are both polynomials in $t_1^d$ and $t_2,\cdots,t_r$, so replacing $t_1^d$ by $t_1$ throughout again keeps us rational/C-finite.
\end{proof}


\begin{lemma}\label{le:divexp}
  Let $H(t_1,\cdots,t_r)\in \bK[[t_1,\cdots,t_r]]$ be a formal power series which
  \begin{enumerate}[(a)]
  \item\label{item:13} is rational, in the sense that it is expressible as
    \begin{equation}\label{eq:ratpq}
      H(t_1,\cdots,t_r) = \frac{P(t_1,\cdots,t_r)}{Q(t_1,\cdots,t_r)}
    \end{equation}
    for polynomials $P,Q\in \bK[t_1,\cdots,t_r]$, and
  \item\label{item:14} is expressible as a formal power series of $t_1^d$ and $t_2,\cdots,t_r$, i.e. contains only monomials in which the exponent of $t_1$ is divisible by $d$, where $d$ is some fixed positive integer. 
  \end{enumerate}
  Then, we can write \Cref{eq:ratpq} for polynomials $P$ and $Q$ in $t_1^d$ and $t_2,\cdots,t_r$.
\end{lemma}
\begin{proof}
  We will assume that we have an expression \Cref{eq:ratpq} for {\it coprime} $P$ and $Q$, and seek to show that they are polynomials in $t_1^d$ and $t_2,\cdots,t_r$.

  Since we can induct on the number of prime divisors of $d$, we may as well assume that the latter is prime to begin with. There are now two cases to treat:

  {\bf Case 1: The prime $d$ is not $\mathrm{char}~\bK$.} Let $\zeta\in\overline{\bK}$ be a primitive $d^{th}$ root of unity (one exists, precisely because $d\ne \mathrm{char}~\bK$). Condition \Cref{item:14} says that
  \begin{equation*}
    H(\zeta t_1,t_2,\cdots,t_r) = H(t_1,t_2,\cdots,t_r),
  \end{equation*}
  and hence the same holds for $\frac{P}{Q}$. Now, the coprimality of $P$ and $Q$ and the fact that $Q$ has non-vanishing free term \cite[Theorem 1]{ges} imply that 
  \begin{equation*}
    P(\zeta t_1,t_2\cdots,t_r) = P(t_1,t_2\cdots,t_r)\quad\text{and}\quad Q(\zeta t_1,t_2\cdots,t_r) = Q(t_1,t_2\cdots,t_r). 
  \end{equation*}
  In turn, this is precisely the desired conclusion.   
  
  {\bf Case 2: $d=\mathrm{char}~\bK$.} This time around condition \Cref{item:14} is expressible as
  \begin{equation*}
    \frac{\partial H}{\partial t_1} = 0\Rightarrow \frac{\partial}{\partial t_1}\left(\frac{P}{Q}\right) = 0.
  \end{equation*}
  We can henceforth ignore $H$ and work only with polynomials, which we will evaluate at tuples of elements in the algebraic closure $\overline{\bK}$.

  We can evaluate $P$ and $Q$ at some tuple $(t_2,\cdots,t_r)\in \overline{\bK}^{r-1}$ so as to ensure that
  \begin{equation*}
    p(t):=P(t,t_2,\cdots,t_{r})\quad\text{and}\quad q(t):=Q(t,t_2,\cdots,t_{r})
  \end{equation*}
  are coprime. Our goal is now to show that if $\frac pq$ has vanishing formal derivative (with respect to $t$), then $p$ and $q$ are both polynomials in $t^d$. Write
  \begin{equation*}
    p(t) = (t-\lambda_1)^{m_1}\cdots (t-\lambda_l)^{m_l}
  \end{equation*}
  and 
  \begin{equation*}
    q(t) = (t-\mu_1)^{n_1}\cdots (t-\mu_k)^{n_k}
  \end{equation*}
  for distinct $\lambda_i$ and $\mu_j$ in $\overline{\bK}$. Any $m_i$ and $n_j$ that are divisible by $d$ can be eliminated, since for any polynomial $u$ we have
  \begin{equation*}
    \frac{d}{dt}\left(u^d \frac pq\right) =0\iff \frac{d}{dt}\left(\frac pq\right) =0;
  \end{equation*}
  in other words, we may as well assume that {\it all} $m_i$ and $n_j$ are coprime to $d$. If there is at least one numerator factor $t-\lambda_1$, the derivative $\left(\frac pq\right)'$ will have vanishing order $m_1-1$ at $\lambda_1$, and hence not vanish. One argues similarly for the denominator factors, concluding that the original $p$ and $q$ must have been $d^{th}$ powers (and hence polynomials in $d$, since $d=\mathrm{char}~\bK$) to begin with.
\end{proof}

\begin{remark}
  Parts \Cref{item:9} and \Cref{item:matpr} of \Cref{th:prodisrat} are very much in the spirit of \cite[Theorem 3.8 (vi) and (viii)]{lip} respectively, which are the analogous results for {\it $D$-finite} (rather than rational) power series.

  Parts \Cref{item:12} and \Cref{item:8} (which, although stated only for the index $n_1$ for brevity, have obvious variants valid for the other $n_i$) are multi-variable analogues of \cite[Theorem 4.2, parts 3. and 4.]{kp} respectively.
\end{remark}

The product ${\bf a}\bullet{\bf b}$ in \Cref{th:prodisrat}, \Cref{item:7} is what is usually referred to as the {\it Hadamard product}: see e.g. \cite[\S 2.1]{kp} or \cite[discussion preceding Proposition 6.1.11]{enum2} (we apply the term freely to both multi-sequences and their corresponding power series). Given
\begin{itemize}
\item \Cref{th:prodisrat}, \Cref{item:7}, which ensures the rationality of the Hadamard product if one of the factors is C-finite;
\item which specializes to the well-known fact that for plain, 1-dimensional sequences rationality is closed under Hadamard products \cite[Theorem 4.2 2.]{kp};
\item while at the same time the Hadamard product of {\it algebraic} power series need not be algebraic (\cite[discussion immediately preceding \S 6.5]{kp} and \cite[paragraph preceding Proposition 6.1.11]{enum2}),
\end{itemize}

one might naturally ask whether the Hadamard product of two rational multi-variable power series is again rational. This is not true in general, in more than one variable. To place the example in context, recall (\cite[Definition 6.3.1]{enum2}):

\begin{definition}\label{def:diag}
  The {\it diagonal} of a multi-variable power series
  \begin{equation*}
    H(t_1,\cdots,t_r) = \sum_{n_i} a_{n_1\cdots n_r} t_1^{n_1}\cdots t_r^{n_r}
  \end{equation*}
  is the series
  \begin{equation*}
    \mathrm{diag}~H(t):=\sum_n a_{n\cdots n}t^n. 
  \end{equation*}
  The same terminology applies to (multi-)sequences: the diagonal of the multi-sequence
  \begin{equation*}
    {\bf a} = (a_{\bf n})_{{\bf n}=(n_1,\cdots,n_r)}
  \end{equation*}
  is the sequence
  \begin{equation*}
    \mathrm{diag}~{\bf a}:=(a_{n,\cdots,n})_n.
  \end{equation*}
\end{definition}

Picking out the constituent terms $a_{n\cdots n}$ of the diagonal sequence attached to ${\bf a}$ can be achieved by forming the Hadamard product
\begin{equation*}
  H_{\bf a}(t_1,\cdots,t_r) \bullet \frac 1{1-t_1\cdots t_r}.
\end{equation*}
Since diagonals of rational power series need not be rational (\cite[Example 6.3.2]{enum2}), this allows us to construct rational series with non-rational Hadamard product (see also \cite[p.403]{sw-alg}).

\begin{example}
  Consider the rational multi-sequences ${\bf a}$ and ${\bf b}$ with Hilbert series
  \begin{equation*}
    H_{\bf a}(s,t)=\frac 1{1-s-t}\quad\text{and}\quad H_{\bf b}(s,t)=\frac 1{1-st}.
  \end{equation*}
  As follows from the aforementioned \cite[Example 6.3.2]{enum2}, we have
  \begin{equation*}
    H_{{\bf a}\bullet {\bf b}}(s,t) = \sum_n \tbinom{2n}{n} s^n t^n = \frac 1{\sqrt{1-4st}},
  \end{equation*}
  which is not rational.
\end{example}

\subsection{Recursive polynomial sequences}\label{subse:recpolys}

We will need a ``composition'' operation between sequences of polynomials and plain complex number sequences, as detailed below.

\begin{definition}\label{def:op}
  For a polynomial
  \begin{equation}\label{eq:poly}
    P(x)=\sum_k c_k x^k
  \end{equation}
  and a sequence ${\bf a}=(a_n)_n$ we write
  \begin{equation*}
    P\triangleright{\bf a} = \sum_k c_k a_k.
  \end{equation*}
  Similarly, for a sequence  $\cP=(P_n)_n$ of polynomials and a complex number sequence ${\bf a}=(a_n)_n$ we write $\cP\triangleright {\bf a}$ for the sequence $(P_n\triangleright{\bf a})_n$.
\end{definition}
In other words, one simply substitutes $a_k$ for $x^k$ in the polynomials $P_n$ and evaluates to obtain the $n^{th}$ term $b_n$.

\begin{definition}\label{def:rec-comp}
  With $\cP=(P_n)_n$ and $(a_n)_n$ as in \Cref{def:op} we say that the two sequences $\cP$ and ${\bf a}$ are {\it recursively compatible} if the sequence $\cP\triangleright{\bf a}$ is eventually linear recursive.

  The sequence $\cP$ of polynomials is {\it recursively well-adjusted} (or just `well-adjusted' for short) if it is recursively compatible with every eventually linear recursive sequence ${\bf a}=(a_n)_n$. 
\end{definition}

Note that the operation
\begin{equation*}
  (\cP,{\bf a})\mapsto \cP\triangleright {\bf a}
\end{equation*}
is additive in both variables, and hence so is the recursive compatibility relation.




\begin{proposition}\label{pr:qpn}
  Every eventually linear recursive polynomial sequence $\cP=(P_n)_n$ is recursively well-adjusted in the sense of \Cref{def:rec-comp}.
\end{proposition}
\begin{proof}
  Given
  \begin{itemize}
  \item the additivity of recursive compatibility noted just before the statement,
  \item the characterization of eventually recursive sequences in \labelcref{item:3} of \Cref{th:st-rec} above,
  \item and the fact that we can ignore finitely many initial sequence terms $a_i$, $0\le i\le k$ by \Cref{le:ign} below,
  \end{itemize}
  it is enough to prove that $\cP$ is recursively compatible with the sequence ${\bf a}=(a_n)_n$ given by
\begin{equation*}
  a_n = n^d \gamma^n
\end{equation*}
for some non-negative integer $d$ and some $\gamma\in \bC$.

First, note that when $d=0$ and hence $a_n=\gamma^n$ we have
\begin{equation*}
  \cP\triangleright{\bf a} = (P_n(\gamma))_n
\end{equation*}
and hence the conclusion is immediate. In general, consider a linear recurrence of $\cP$ (for large $n$):
\begin{equation}\label{eq:recpol}
  P_{n+T}=c_{T-1}P_{n+T-1}+\cdots+c_0 P_n
\end{equation}
for polynomials $c_i$. Then, the derived polynomials satisfy the relation
\begin{equation}\label{eq:6}
  P'_{n+T}=c_{T-1}P'_{n+T-1}+\cdots+c_0 P'_n + c'_{T-1}P_{n+T-1}+\cdots+c'_0 P_n.
\end{equation}
If $d=1$ then the conclusion amounts to proving that $(P'_n(\gamma))_n$ is eventually recursive; this, in turn, follows from \Cref{eq:6} and the fact that $(P_n(\gamma))_n$ is (eventually) recursive. 

For $d=2$ repeat the procedure: \Cref{eq:6} once more shows that $(P'_n(\gamma))_n$ is recurrent (as, of course, is $(P_n(\gamma))_n$). Differentiating once more we obtain
\begin{equation*}
  P''_{n+T} = \sum_{i=0}^{T-1} (c_i P''_{n+i} + 2 c'_i P'_{n+i} + c''_i P_{n+i}).
\end{equation*}
This, in turn, shows that $(P''_n(\gamma))_n$ is recurrent and hence so is $\cP\triangleright (n^2 \gamma^n)_n$. 

It should be clear now how to continue this recursive process to conclude for arbitrary $d$.
\end{proof}

\begin{proof}[alternative]
  The fact that $(P_n(x))_n$ is eventually linearly recursive implies that its Hilbert series
  \begin{equation*}
    H_{\cP}(x,y) = \sum_{m,n}c_{m,n}x^m y^n:=\sum_{n\ge 0} P_n(x)y^n,
  \end{equation*}
  which a priori is an element of $\bK[x][[y]]$ (formal $y$-power series over the polynomial ring in $x$), is rational:
  \begin{equation*}
    H_{\cP}(x,y) = \frac{A(x,y)}{B(x,y)},\quad A,B\in \bK[x,y]. 
  \end{equation*}
  The Hilbert series $H_{\cP\triangleright {\bf a}}(y)$ of $\cP\triangleright {\bf a}$ is obtained from $H_{\cP}(x,y)$ by substituting $a_m$ for each $x^m$; in other words, the coefficient of $y^n$ in $H_{\cP\triangleright{\bf a}}$ is
  \begin{equation*}
    \sum_m c_{m,n} a_m.
  \end{equation*}
  The fact that the resulting power series is rational (and hence $\cP\triangleright{\bf a}$ is eventually linearly recursive by \Cref{th:st-rec}) follows from part \Cref{item:matpr} of \Cref{th:prodisrat}.
\end{proof}

\begin{lemma}\label{le:ign}
  Under the hypotheses of \Cref{pr:qpn}, $\cP\triangleright{\bf a}$ is eventually linearly recursive if ${\bf a}$ eventually vanishes.
\end{lemma}
\begin{proof}
  This is almost immediate: if $a_{i}=0$ for $i>k$ then the $n^{th}$ term of $\cP\triangleright{\bf a}$ is a linear combination (with constant coefficients) of the first $k+1$ coefficients of $P_n$, and a recurrence relation \Cref{eq:recpol} induces one for each coefficient $a_i$, $0\le i\le k$.
\end{proof}

\subsection{Laurent polynomials}\label{subse:spec}

It will be useful later on, in \Cref{se.omega}, to have a Laurent-polynomial analogue of sorts for \Cref{pr:qpn}. To elaborate, we will have
\begin{itemize}
\item a linearly recursive sequence $\cP=(P_n)_n$ of {\it Laurent} polynomials;
\item and C-finite sequences $({\bf a}_n)$ and $({\bf b}_n)$;
\item and the goal of showing that $\cP\triangleright({\bf a},{\bf b})$ is well-behaved (C-finite, algebraic, D-finite, etc.),
\end{itemize}
the latter symbol is the object of the following expansion of \Cref{def:op}. 

\begin{definition}\label{def:laurtr}
  For a Laurent polynomial \Cref{eq:poly} and sequences ${\bf a}=(a_n)_{n\in \bN}$ and ${\bf b}=(b_n)_{n\in \bN}$ we write
  \begin{equation*}
    P\triangleright({\bf a},{\bf b}) = \sum_{k\ge 0} c_k a_k + \sum_{k<0} c_k b_{-k-1}. 
  \end{equation*}
  In other words, $\cP\triangleright({\bf a},{\bf b})$ is defined similarly to $\cP\triangleright{\bf a}$, except this time around
\begin{itemize}
\item $a_k$ is substituted for each $x^k$ appearing in $P$,
\item while $b_k$ is substituted for each $x^{-k-1}$ appearing in $P$. 
\end{itemize}
For a sequence $\cP=(P_n)_n$ of Laurent polynomials we write
\begin{equation*}
  \cP\triangleright({\bf a},{\bf b}) := (P_n\triangleright({\bf a},{\bf b}))_n. 
\end{equation*}
\end{definition}

We can now state

\begin{theorem}\label{th:pab}
  Let $\cP=(P_n)_n$ be an eventually linearly recursive sequence of Laurent polynomials and ${\bf a}$, ${\bf b}$ two eventually linearly recursive sequences. Then, $\cP\triangleright({\bf a},{\bf b})$ is algebraic.
\end{theorem}
\begin{proof}
Since
\begin{equation*}
  \cP\triangleright({\bf a},{\bf b}) = \cP\triangleright({\bf a},{\bf 0}) + \cP\triangleright({\bf 0},{\bf b}),
\end{equation*}
it is enough to work with a single sequence ${\bf a}$ and hence try to argue that
\begin{equation*}
  \cP\triangleright{\bf a} := \cP\triangleright({\bf a},{\bf 0})
\end{equation*}
is algebraic. This means that we are substituting $a_n$ for the non-negative-exponent $x^n$ appearing in the $P_m$, and dropping the negative-exponent $x^{-n-1}$, $n\in \bN$.

In this setting, we can replace Laurent with ordinary polynomials at the cost of replacing `$\triangleright$' with a more sophisticated operation. To see this, first consider a recursion \Cref{eq:recpol}
\begin{equation*}
  P_{n+T}(x)=\sum_{i=1}^T c_{T-i}(x)P_{n+T-i}(x),
\end{equation*}
holding for all $n$ sufficiently large (say for $n\geq \bar{N}$), where $c_{\bullet}(x)$ are Laurent polynomials and $T\geq 1$ is a positive integer. Choosing natural numbers $A$ and $B$ such that $x^{Ai}c_{T-i}(x)$ and $x^{Aj+B}P_j(x)$ are polynomials for $1\leq i\leq T$ and $0\leq j<\bar{N}+T$, we have
\begin{align*}
x^{A(n+T)+B}P_{n+T}(x)&=\sum_{i=1}^T x^{A(n+T)+B} c_{T-i}(x) P_{n+T-i}(x)\\
&=\sum_{i=1}^T x^{Ai}c_{T-i}(x)(x^{A(n+T-i)+B} P_{n+T-i}(x))\\
\end{align*}
for $n\geq \bar{N}$. Replacing the original $P$ with the {\it polynomials} $x^{An+B}P_{n}(x)$ satisfying a linear recurrence with respective polynomial coefficients $x^{Ai}c_{T-i}(x)$ in place of $c_{T-i}$, we may as well assume that everything in sight is a plain (as opposed to Laurent) polynomial.

The substitution of terms $a_n$ for powers of $x$ in $P_m$, though, now takes on a different character. We will have some polynomial $q(n)=An+B$ such that
\begin{itemize}
\item we substitute $a_0$ for each $x^{q(n)}$ in each $P_n(x)$ (note the correlation: as $n$ grows, we start substituting $a$s for $x$s in $P_n$ starting with larger and larger exponents $q(n)$);  
\item similarly, we substitute $a_1$ for each $x^{q(n)+1}$;
\item etc.
\end{itemize}

The recursion \Cref{eq:recpol} shows that
\begin{equation*}
  \deg P_{n+T} \le \max_{0\le i\le T-1} \deg c_i P_{n+i},
\end{equation*}
which puts a bound of $Dn$ (for fixed $D$) on the degree of $P_n$. We may thus assume that the substitution of $a$s for $x$s takes place over a range of exponents for monomials of $P_n$: starting with $x^{q(n)}=x^{An+B}$ and ending with $x^{Dn}$.

We can now proceed along the lines of the alternative proof of \Cref{pr:qpn}:
\begin{itemize}
\item consider the rational Hilbert series
\begin{equation*}
  H_{\cP}(x,y) = \sum_{m,n}c_{m,n}x^m y^n:=\sum_{n\ge 0} P_n(x)y^n
\end{equation*}
with its attached rational 2-sequence $(c_{m,n})$;
\item note that it will make no difference to change finitely many members of ${\bf a}$, because the difference to the original sequence would then be eventually vanishing, and the problem would reduce to arguing that the sequences
  \begin{equation*}
    (b_{q(n),n})_n,\ (b_{q(n)+1,n})_n,\cdots,(b_{q(n)+\ell,n})_n
  \end{equation*}
  are algebraic, for fixed $\ell$. In turn, this follows from the fact that the diagonal of a rational 2-sequence is algebraic \cite[Theorem 6.3.3]{enum2}.
\item but then we may as well assume that ${\bf a}$ is of the form
  \begin{equation*}
    a_n = \sum_{i=1}^s Q_i(n)\gamma_i^n
  \end{equation*}
  for {\it all} $n$, and hence is extendable to negative $n$ by the same formula, and then also extendable to the C-finite 2-sequence
  \begin{equation*}
    (a_{m-q(n)})_{m,n} = \sum_{i=1}^s Q_i(m-q(n))\gamma_i^{m-q(n)};
  \end{equation*}
  the C-finiteness follows because the 2-sequence has the shape described in part \Cref{item:10} of \Cref{th:multirat}.
\item now form the (also rational, by \Cref{th:prodisrat} \Cref{item:7}) 2-sequence $a'_{m,n}:=c_{m,n}a_{m-q(n)}$;
\item which then yields a 2-sequence
  \begin{equation*}
    b_{m,n} := a'_{q(m),n} + a'_{q(m)+1,n}+\cdots + a'_{Dm,n},
  \end{equation*}
  rational by parts \Cref{item:12} and \Cref{item:8} of \Cref{th:prodisrat};
\item which in turn has an algebraic {\it diagonal} sequence
  \begin{equation*}
    b_n:=b_{n,n}=a'_{q(n),n} + a'_{q(n)+1,n}+\cdots + a'_{Dn,n}
  \end{equation*}
  by \cite[Theorem 6.3.3]{enum2}. 
\end{itemize}
$b_n$ is our target sequence, and the conclusion that it is algebraic is precisely what we were after.
\end{proof}

An example will illustrate the substitution $x^n\to a_n$ in the discussion above.

\begin{example}\label{ex:catagain}

Take $P_n=(\frac{1}{x}+x)^n$ and $a_n=\delta_{0,n}$ (i.e. ${\bf a}=(1,0,0,\cdots)$). Furthermore, take $q(n)=n$. Then for all $n$, 
\begin{equation*}x^nP_n=(1+x^2)^n.
\end{equation*}  We consider $\mathcal{P}\triangleright({\bf a},{\bf 0})$. The substitution in question will pick out the coefficient of $x^n$ in $x^nP_n$, i.e. will return the sequence
  \begin{equation*}
    b_n=
    \begin{cases}
      \tbinom{n}{\frac n2}&\text{ if $n$ is even}\\
      0&\text{ otherwise}
    \end{cases}
  \end{equation*}
  This is an algebraic sequence: by \cite[Example 6.3.2]{enum2}, its Hilbert series is $\frac 1{\sqrt{1-4x^2}}$.\\~\\
\end{example}
\section{Results Concerning $(P_n\triangleright {\bf a})_n$}\label{se:conv}
\Cref{def:op} can be extended to define $P\triangleright {\bf a}$ where $P$ is a polynomial of $\ell$ variables, and ${\bf a}$ is an $\ell$-sequence.
Let
\begin{equation*}P(x_1, ..., x_\ell)=\sum_{s_1, ..., s_\ell}c_{ s_1, ..., s_\ell}x_1^{s_1}\cdots x_\ell^{s_\ell},
\end{equation*}
and let ${\bf a}=(a_{i_1, ..., i_\ell})_{i_1, ..., i_\ell}$ be an $\ell$-sequence. Then 
\begin{equation*}P\triangleright{\bf a}=\sum_{s_1, ..., s_\ell}c_{ s_1, ..., s_\ell}a_{{s_1},...,{s_\ell}}.
\end{equation*}
Further, if $\mathcal{P}=(P_n)_n$ then $\mathcal{P}\triangleright{\bf a}=(P_n\triangleright {\bf a})_n$.
\\~\\
In this section, we consider $(P_n\triangleright {\bf a})_n$ where, unless otherwise stated, $(P_n)_n$ is an eventually linear recursive sequence of complex polynomials (of several variables) and ${\bf a}$ is a multi-sequence of complex numbers. In particular, we consider the sequence $(P_n\triangleright {\bf a})_n$ where ${\bf a}$ has a given property (rational, algebraic, P-recursive).  

\subsection{P-recursive multi-sequences}
Recall the definition of P-recursive (multi-) sequences as given in \cite[Definition 3.2]{lip} and as restated below: 
\begin{definition}
A sequence $a(i_1, i_2,...,i_\ell)$ is P-recursive if there is a natural number $m$ such that 
\begin{enumerate}
\item For each $j=1, 2,...,\ell$ and each ${\bf v}=(v_1, v_2,...,v_\ell)\in \{0, 1, ..., m\}^\ell$ there is a polynomial $p_{\bf v}^{(j)}$ (with at least one $p_{\bf v}^{(j)}\not=0$ for each $j$) such that 
\begin{equation*}\sum_{\bf v} p_{\bf v}^{(j)}(i_j)\; a(i_1-v_1, i_2-v_2,...,i_\ell-v_\ell) =0
\end{equation*}
for all $i_1, i_2,...,i_\ell\geq m$,  and 
\item if $\ell>1$ then all the $m$-sections of $a(i_1, i_2,...,i_\ell)$ are P-recursive. 
\end{enumerate}
\end{definition}
We find that if $(P_n)_n$ is an eventually linear recursive sequence of polynomials and ${\bf a}$ is a $P$-recursive multi-sequence, then $(P_n\triangleright {\bf a})_n$ is a P-recursive sequence, as stated in \Cref{precursivechar0}. To prove this result, we will require the following theorem from \cite[Theorem 3.8, (i) and (vi)]{lip}. 
\begin{theorem}\label{thm3.8}
\begin{enumerate}[(a)]
\item The P-recursive sequences (of dimension $\ell$) form an algebra over $\mathbb{C}[i_1, i_2,...,i_\ell]$.
\item  If $(a_{i_1,...,i_\ell})_{i_1,...,i_\ell}$ is P-recursive and $\sum_{i_\ell} a_{i_1, ..., i_\ell}$ converges for every $i_1, ..., i_{\ell-1}$ then the sequence $(b_{i_1,...,i_{\ell-1}})_{i_1,...,i_{\ell-1}}$ given by
\begin{equation*}b_{i_1, ..., i_{\ell-1}}=\sum_{i_\ell}a_{i_1, ..., i_{\ell-1}, i_\ell}
\end{equation*} is P-recursive. 
\end{enumerate}
\end{theorem}

\begin{proposition} \label{precursivechar0} Let $(P_n)_n$ be an eventually linear recursive sequence of complex polynomials and let ${\bf a}:\mathbb{N}^\ell\rightarrow \mathbb{C}$ be a P-recursive multi-sequence. Then $\left(P_n\triangleright {\bf a}\right)_n$
is a P-recursive sequence. 
\end{proposition}

\begin{proof}Let $\mathcal{P}=(P_n(x_1, ..., x_\ell))_n$ be an eventually linear recursive sequence of polynomials. For each $n$, we write
\begin{equation*} P_n(x_1, ..., x_\ell)=\sum_{s_1, ..., s_\ell} c_{n, s_1, ..., s_\ell} x_1^{s_1} \cdots x_\ell^{s_\ell}.
\end{equation*}
Since $(P_n(x_1, ..., x_\ell))_n$ is an eventually linear recursive sequence, by Theorem \ref{th:st-rec}(a), 
\begin{equation*}H_{\mathcal{P}}(t)=\sum_nP_n(x_1, ..., x_\ell)t^n=\sum_{n,s_1,..., s_\ell} c_{n, s_1, ..., s_\ell} x_1^{s_1} \cdots x_\ell^{s_\ell} t^n 
\end{equation*}
is rational. It follows that the attached multi-sequence $(c_{n, s_1,  ..., s_\ell})_{n, s_1, ..., s_\ell}$ is rational and therefore a P-recursive multi-sequence (by \cite[Proposition 2.3, (ii)]{lip}). 
\\~\\
Let ${\bf a}=(a_{i_1,...,i_\ell})_{i_1, ..., i_\ell}$ be a P-recursive multi-sequence. For each $n$, define 
\begin{equation*}\bar{a}_{n, i_1,  ..., i_\ell}=a_{i_1,...,i_\ell}.
\end{equation*}
Since $(a_{i_1,...,i_\ell})_{i_1, ..., i_\ell}$ is P-recursive, $(\bar{a}_{n,i_1, ..., i_\ell})_{n,i_1, ..., i_\ell}$ is also P-recursive. 
Then for each $n$,
\begin{equation*}P_n\triangleright {\bf a}=\sum_{i_1,  ..., i_\ell} c_{n,i_1, ..., i_\ell} a_{i_1,  ..., i_\ell}=\sum_{i_1,  ..., i_\ell} c_{n,i_1,  ..., i_\ell} \bar{a}_{n, i_1, ..., i_\ell}.
\end{equation*}
By \Cref{thm3.8} part (a), $\left(c_{n,i_1, ..., i_\ell} \bar{a}_{n, i_1,  ..., i_\ell}\right)_{n, i_1, ..., i_\ell}$ is a P-recursive multi-sequence. Since $(c_{n, s_1, ..., s_\ell})_{s_1, ..., s_\ell}$ are the coefficients of the terms of $P_n(x_1,  ..., x_\ell)$, for a fixed $n$ there are finitely many $s_1, ..., s_\ell$ such that $c_{n,s_1, ..., s_\ell}\not=0$. Thus, by applying \Cref{thm3.8} part (b) a number of times, we find that
\begin{equation*}\left(\sum_{i_1, ..., i_\ell}c_{n,i_1, ..., i_\ell}\bar{a}_{n,i_1, ..., i_\ell}\right)_{n}=\left(\sum_{i_1, ..., i_\ell}c_{n, i_1,  ..., i_\ell}a_{i_1,...,i_\ell}\right)_{n}
\end{equation*}
is a P-recursive sequence. 
That is, $\left(P_n\triangleright {\bf a}\right)_n$ is a P-recursive sequence. 
\end{proof}

\subsection{Examples}
When ${\bf a}$ is a rational or algebraic multi-sequence, $(P_n\triangleright{\bf a})_n$ is P-recursive by \Cref{precursivechar0}. We wonder if this result could be improved: If ${\bf a}$ is a rational (or algebraic) multi-sequence, is $(P_n\triangleright {\bf a})_n$ necessarily rational (or algebraic)? In the following examples we see that this is not always the case over $\mathbb{C}$. 

\begin{example}
 If  ${\bf a}:\mathbb{N}^2\rightarrow \mathbb{C}$ is a rational multi-sequence, and $(P_n(x_1,x_2))_n$ is a linear recursive sequence of complex polynomials then $(P_n\triangleright {\bf a})_n$ may not be a rational sequence.  
\\~\\
Let ${\bf a}=(a_{i_1, i_2})_{i_1, i_2}$ be given by: 
\begin{equation*}a_{i_1, i_2}=\begin{cases} {1} &\text{ for } i_1=i_2\\ 0& \text{ else}\end{cases}
\end{equation*}
Then, 
\begin{equation*}H_{{\bf a}}(x_1, x_2):=\sum_{i_1, i_2} a_{i_1, i_2}x_1^{i_1}x_2^{i_2}=\sum_i x_1^i x_2^i=\dfrac{1}{{1-x_1x_2}}
\end{equation*}
and $(a_{i_1, i_2})_{i_1, i_2}$ is a rational multi-sequence.
\\~\\
Let $P_n(x_1, x_2)=(x_1+x_2)^{2n}$. That is, $P_n(x_1, x_2)=(x_1+x_2)^2P_{n-1}(x_1, x_2)$ and $(P_n(x_1, x_2))_n$ is a linear recursive sequence. Notice that, 

\begin{equation*}H_{\mathcal{P}\triangleright{\bf a}}(t):=\sum_n \left(P_n\triangleright {\bf a} \right)t^n=\sum_n {2n\choose n} t^n=\dfrac{1}{\sqrt{1-4t}}
\end{equation*}
as in \cite[Example 6.3.2]{enum2}. Therefore, in this case, $(P_n\triangleright {\bf a})_n$ is not a rational sequence. 
\end{example}

\begin{example}\label{example2}
 Let ${\bf a}: \mathbb{N}^\ell\rightarrow \mathbb{C}$ be a rational multi-sequence, and $\mathcal{P}=(P_n(x_1, ..., x_\ell))_n$ be a linear recursive sequence of complex polynomials. If $\ell>2$, then $(P_n\triangleright {\bf a})_n$ may not be an algebraic sequence. 
\\~\\
Let $\ell>2$ be a natural number. For each $n$, let $P_n(x_1, ..., x_\ell)=(x_1+\cdots+x_\ell)^{\ell n}$. Then \begin{equation*}P_n(x_1, ..., x_\ell)=(x_1+\cdots+x_\ell)^\ell P_{n-1}(x_1,  ..., x_\ell)
\end{equation*} and $(P_n(x_1, ..., x_\ell))_n$ is a linear recursive sequence. 
\\~\\
Let ${\bf a}=(a_{i_1, ..., i_\ell})_{i_1, ..., i_\ell}$ be defined by: 
\begin{equation*}a_{i_1, ..., i_\ell}=\begin{cases} 1 & \text{ for } i_1=\cdots= i_\ell\\ 0 &\text{ else} \end{cases}.
\end{equation*}
Then 
\begin{equation*}H_{\bf a} (x_1,  ..., x_\ell):=\sum_{i_1, ..., i_\ell}a_{i_1,..., i_\ell}x_1^{i_1}\cdots x_\ell^{i_\ell}=\sum_{i}x_1^{i}\cdots x_\ell^{i}=\dfrac{1}{1-x_1\cdots x_\ell}
\end{equation*}
and $(a_{i_1,  ..., i_\ell})_{i_1, i_2, ..., i_\ell}$ is a rational multi-sequence. Notice that 
\begin{equation*}H_{\mathcal{P}\triangleright{\bf a}}(t):=\sum_n \left(P_n \triangleright {\bf a}\right) t^n=\sum_n {\ell n\choose n,n,...,n}t^{n}.
\end{equation*}
However, this series is transcendental over any field characteristic zero (see, for example \cite[Theorem 3.8]{WoodcockSharifTrans}). Therefore, in this case, $(P_n\triangleright {\bf a})_n$ is not an algebraic sequence. 
\end{example}

\begin{example}\label{example3} If ${\bf a}:\mathbb{N}^2\rightarrow \mathbb{C}$ is an algebraic multi-sequence and $(P_n(x_1,x_2))_n$ is a linear recursive sequence of complex polynomials, $(P_n\triangleright {\bf a})_n$ may not be algebraic. 
\\~\\
Let ${\bf a}=(a_{i_1, i_2})_{i_1, i_2}$ be given by: 
\begin{equation*}a_{i_1, i_2}=\begin{cases} {2i_1\choose i_1} &\text{ for } i_1=i_2\\ 0& \text{ else}\end{cases}
\end{equation*}
Then, 
\begin{equation*}H_{{\bf a}}(x_1, x_2):=\sum_{i_1, i_2} a_{i_1, i_2}x_1^{i_1}x_2^{i_2}=\sum_i {2i\choose i}x_1^i x_2^i=\dfrac{1}{\sqrt{1-4x_1x_2}}
\end{equation*}
as in \cite[Example 6.3.2]{enum2}. That is, $(a_{i_1, i_2})_{i_1, i_2}$ is an algebraic multi-sequence (but not rational).
\\~\\
Let $P_n(x_1, x_2)=(x_1+x_2)^{2n}$. Then $P_n(x_1, x_2)=(x_1+x_2)^2P_{n-1}(x_1, x_2)$ and $(P_n(x_1, x_2))_n$ is a linear recursive sequence. Notice that, 

\begin{equation*}H_{\mathcal{P}\triangleright{\bf a}}(t):=\sum_n \left(P_n(x_1, x_2)\triangleright {\bf a}\right) t^n=\sum_n {2n\choose n}^2 t^n
\end{equation*}
which  is transcendental over $\mathbb{C}$ (see for example \cite[\S 4, Example (g)]{RPStanleyDiff}). Therefore, in this case, $(P_n\triangleright{\bf a})_n$ is not an algebraic sequence. 
\end{example}
Over fields of characteristic $p$, examples analogous to \Cref{example2} and \Cref{example3} cannot be found, as stated in Proposition \ref{algebraiccharp}. 
\begin{proposition}\label{algebraiccharp}
Let $\mathbb{K}$ be a field of characteristic $p$. Let ${\bf a}: \mathbb{N}^\ell\rightarrow \mathbb{K}$ be an algebraic multi-sequence and let $(P_n)_n$ be an eventually recursive sequence of polynomials in $\mathbb{K}[x_1, ...,x_\ell]$. Then $(P_n\triangleright {\bf a})_n$ is an algebraic sequence.  
\end{proposition}
The proof of \Cref{algebraiccharp} is similar to the proof of \Cref{precursivechar0}, however instead of applying \Cref{thm3.8}, we require the main theorem in \cite{sw-alg} (as restated in \Cref{swMainThm}), as well as \Cref{algebraicindices}. 
\begin{theorem} \label{swMainThm}
If $\mathbb{K}$ is a field of characteristic $p>0$ and if $f, g$ are algebraic series over $\mathbb{K}$, then the Hadamard product of $f$ and $g$ is again an algebraic series over $\mathbb{K}$. 
\end{theorem}

\begin{proposition}\label{algebraicindices}
If $(a_{i_1, ..., i_\ell})_{i_1, ..., i_\ell}$ is an algebraic multi-sequence and for each $(\ell-1)$-tuple $(i_1, ..., i_{\ell-1})$, $a_{i_1,  ..., i_\ell}\not=0$ for finitely many $i_\ell$, then ${\bf b}=(b_{i_1,  ..., i_{\ell-1}})_{i_1,  ..., i_{\ell-1}}$ given by
\begin{equation*}b_{i_1,  ..., i_{\ell-1}}= \sum_{i_\ell} a_{i_1, ..., i_\ell}
\end{equation*} 
is an algebraic (multi-)sequence. 
\end{proposition}
\begin{proof}
Since ${\bf a}=(a_{i_1,..., i_\ell})_{i_1, ..., i_\ell}$ is an algebraic multi-sequence, 
\begin{equation*}
H_{\bf a}(x_1,..., x_\ell)=\sum_{s_1, ...,s_\ell}a_{s_1, ..., s_\ell}x_1^{s_1}\cdots x_\ell^{s_\ell}
\end{equation*}
satisfies
\begin{equation}\label{algebraiceqn}
P_d(x_1,..., x_\ell)H_{\bf a}(x_1,..., x_\ell)^d+\cdots +P_1(x_1,..., x_\ell)H_{\bf a}(x_1,..., x_\ell)+P_0(x_1,..., x_\ell)=0
\end{equation}
for some $d$, where $\{P_i(x_1,..., x_\ell)\}_{i=0}^d$ are non-vanishing polynomials. Further, we can assume that these polynomials share no common factors; if there were such a common factor, it could be factored from Equation \Cref{algebraiceqn}, yielding an equation in this form where the polynomials do not share a common factor. 
 \\~\\
Since for each $(i_1, ..., i_{\ell-1})$, $a_{i_1, ..., i_\ell}\not=0$ for finitely many $i_\ell$, we consider $H_{\bf a}(x_1, ..., x_{\ell-1}, 1)$ as follows
\begin{equation*}
  H_{\bf a}(x_1, ..., x_{\ell-1}, 1)=\sum_{s_1, ...,s_\ell}a_{s_1, ..., s_\ell}x_1^{s_1}\cdots x_{\ell-1}^{s_{\ell-1}}=\sum_{s_1, ...,s_{\ell-1}}\left( \sum_{s_\ell} a_{s_1, ..., s_\ell}\right) x_1^{s_1}\cdots x_{\ell-1}^{s_{\ell-1}}. 
\end{equation*}
\\~\\
In particular, we notice that $H_{\bf b}(x_1, ..., x_{\ell-1})=H_{\bf a}(x_1, ..., x_{\ell-1}, 1)$. Letting $x_\ell=1$ in Equation \ref{algebraiceqn}, we obtain
\begin{equation*}
P_d(x_1,...,x_{\ell-1}, 1)H_{\bf b}(x_1, ..., x_{\ell-1})^d+\cdots +P_1(x_1,...,x_{\ell-1}, 1)H_{\bf b}(x_1, ..., x_{\ell-1})+P_0(x_1,...,x_{\ell-1}, 1)=0.
\end{equation*}
It is not the case that for all $1\leq i\leq d$, $P_i(x_1, ...,x_{\ell-1}, 1)=0$, since the polynomials $\{P_i(x_1,..., x_\ell)\}_{i=0}^d$ share no common factors, therefore ${\bf b}$ is an algebraic multi-sequence. 
\end{proof}


\section{Polynomial invariants}\label{se.poly}

Since for $\Omega$-algebraic modules $M$ we are reduced to examining the entries of the powers of a matrix over $\bZ[\Omega^{\pm 1}]$, the question arises of whether $\dim \Omega^n M$ is eventually polynomial in $n$. More generally, one can ask this of the ``size'' of $\Omega^n M$ in various guises: dimension, length, length of the socle, etc.

Certainly, $\dim \Omega^nM$ has polynomial {\it growth}: there is a smallest non-negative integer $s$ such that
\begin{equation}\label{eq:1}
  \dim\Omega^nM = O(n^s)
\end{equation}
is standard big-O notation. That $s$ is precisely the {\it complexity} $\mathrm{cx}(M)=\mathrm{cx}_G(M)$, as covered for instance in \cite[\S 2.24]{ben-mod}.

Requiring that $\dim\;\Omega^nM$ be eventually polynomial in $n$ is too much to ask though: when $\mathrm{cx}(M)=1$ the module $M$ is {\it periodic}, in the sense that $\Omega^TM\cong M$ for some $T$ and the sequence is simply periodic. This remark and \cite[Theorem 3.4]{bc-mult} are suggestive of the possibility that perhaps \Cref{eq:1} is always decomposable as a disjoint union of eventually-polynomial sequences. We will see below that this is indeed the case.

Recall the following notion, e.g. from \cite[\S 4.4]{enum1}.

\begin{definition}\label{def.per-poly}
  A sequence $(a_n)$ is {\it quasipolynomial of quasiperiod $T$} if there are polynomials $P_i$, $0\le i\le T-1$ such that
  \begin{equation*}
    a_n = P_{n\;\mathrm{mod}\;T}(n),\ \forall n.
  \end{equation*}
  It is {\it eventually} quasipolynomial if this constraint holds for sufficiently large $n$.
\end{definition}


For a simple $S\in \mathrm{mod}~kG$ and a finite-dimensional $G$-module $M$ we write $\ell_SM$ for the multiplicity of $S$ in $M$ as a composition factor. We then have the following result (essentially contained in \cite[\S 5.3]{bnsn-rep2}).

\begin{proposition}\label{pr.lsoc-poly}
  For a finite group $G$, a finite-dimensional $G$-module $M$ and a simple $G$-module $S$ the sequence
  \begin{equation*}
    n\mapsto \ell_S(\mathrm{soc}\; \Omega^nM) 
  \end{equation*}
  is eventually quasipolynomial in $n$. The same goes for $\Omega^{-n}$ in place of $\Omega^n$.
\end{proposition}
\begin{proof}
  The two versions are interchanged by duality, so it suffices to prove the claim for the cosyzygy functors $\Omega^{-n}$.
  
  According to \cite[Theorem 8.1]{ev-fg} the cohomology
  \begin{equation*}
    \mathrm{Ext}^n(S,M)\cong H^n(G,M\otimes S^*)
  \end{equation*}
  is a finitely generated graded module over the finitely generated skew-commutative graded ring $H^*(G)$. It follows from standard Hilbert-Samuel theory (e.g. \cite[Proposition 5.3.1]{bnsn-rep2}) that the Hilbert series of
  \begin{equation}\label{eq:9}
    n\mapsto \dim\;H^n(G,M\otimes S^*)=\dim\; \mathrm{Ext}^n(S,M)
  \end{equation}
  is of the form $\frac{P(n)}{Q(n)}$ for polynomials $P$ and $Q$ with the zeroes of $Q$ being roots of unity. It then follows from \cite[Proposition 4.4.1]{enum1} that \Cref{eq:9} is eventually quasipolynomial. Since for $n\ge 1$ we have
\begin{equation*}
\dim\; \mathrm{Ext}^n(S,M)=\dim \mathrm{Hom}(S,\Omega^{-n}M) = \text{number of }S\text{ summands of }\mathrm{soc}\;\Omega^{-n}M,
\end{equation*}
this finishes the proof. 
\end{proof}

\begin{corollary}\label{cor.ev-poly}
  Let $G$ be a finite group and $F:\mathrm{mod}~kG\to \mathrm{Vect}^f$ a linear functor. For a finite-dimensional $G$-module $M$ the sequence
  \begin{equation*}
    n\mapsto \dim F(\mathrm{soc}\; \Omega^nM)
  \end{equation*}
  is eventually quasipolynomial in $n$. The same goes for $\Omega^{-n}$ in place of $\Omega^n$.
\end{corollary}
\begin{proof}
  Immediate from \Cref{pr.lsoc-poly}, given that
  \begin{equation*}
    F(\mathrm{soc}\; \Omega^nM) \cong \bigoplus_{\text{simple }S}F(S)^{\oplus \ell_S(\mathrm{soc}\; \Omega^nM)}
  \end{equation*}
  and hence
  \begin{equation*}
    \dim F(\mathrm{soc}\; \Omega^nM) = \sum_{\text{simple }S}\ell_s \dim F(S). 
  \end{equation*}
\end{proof}

We also have the following version, for $\Omega^n M$ rather than their socles. 

\begin{theorem}\label{th.f-per}
  Let $G$ be a finite group and $F:\mathrm{mod}~kG\to \mathrm{Vect}^f$ an exact functor. For a finite-dimensional $G$-module $M$ the sequence
  \begin{equation*}
    n\mapsto \dim F(\Omega^nM)
  \end{equation*}
  is eventually quasipolynomial in $n$. The same goes for $\Omega^{-n}$ in place of $\Omega^n$.  
\end{theorem}
\begin{proof}
  For variety, we focus on $\Omega^{-n}$ this time around.

  Consider a minimal injective resolution
  \begin{equation}\label{eq:7}
    0\to M\to I_0\to I_1\to \cdots
  \end{equation}
  As argued in \cite[\S 5.3]{bnsn-rep2}, for each simple $S$ the multiplicity $m_{S,n}$ of its injective hull $I_S$ as a summand of $I_n$ has a Hilbert series as in the proof of \Cref{pr.lsoc-poly}: rational, with root-of-unity poles. It once more follows from \cite[Proposition 4.4.1]{enum1} that $n\mapsto m_{S,n}$ is eventually quasipolynomial, and hence so is
  \begin{equation}\label{eq:10}
    n\mapsto \dim FI_n = \sum_{\text{simple }S}m_{S,n}\dim FI_S. 
  \end{equation}


  Applying the exact functor $F$ to \Cref{eq:7} produces a long exact sequence,
  \begin{equation*}
    0\to FM\to FI_0\to FI_1\to \cdots
  \end{equation*}
  resulting from splicing together the short exact sequences
  \begin{equation*}
    0\to F\Omega^{-n+1}M \to FI_{n-1} \to F\Omega^{-n}M\to 0,\ n\ge 1.
  \end{equation*}
  These short exact sequences in turn imply that
  \begin{equation*}
    \dim F\Omega^{-n}M = \dim FI_{n-1}-\dim FI_{n-2}+\dim FI_{n-3}-\cdots +(-1)^{n} \dim FM
  \end{equation*}
  (note that the signs alternate).

  We thus obtain
  \begin{equation*}
    \dim F\Omega^{-(n+2)}M-\dim F\Omega^{-n}M = \dim FI_{n+1}-\dim FI_{n},
  \end{equation*}
  and hence the conclusion follows from the quasipolynomial character of \Cref{eq:10}. 
\end{proof}

As an immediate consequence, we have the announced result on dimensions:

\begin{corollary}\label{cor.dim}
  For $G$ and $M$ as in \Cref{th.f-per} the sequence
  \begin{equation*}
    n\mapsto \dim \Omega^nM
  \end{equation*}
  is eventually quasipolynomial in $n$, and similarly for $\Omega^{-n}$
\end{corollary}
\begin{proof}
  Simply take $F$ of \Cref{th.f-per} to be the forgetful functor from $G$-modules to vector spaces.
\end{proof}

The same goes for lengths rather than dimensions: 

\begin{corollary}\label{cor.l}
  For $G$ and $M$ as in \Cref{th.f-per} and a simple module $S\in \mathrm{mod}~kG$ the sequence
  \begin{equation*}
    n\mapsto \ell_S(\Omega^nM)
  \end{equation*}
  is eventually quasipolynomial in $n$, and similarly for $\Omega^{-n}$
\end{corollary}
\begin{proof}
  This is an application of \Cref{th.f-per} with
  \begin{equation*}
    F=\mathrm{Hom}_G(P_S,-)\cong \ell_S(-),
  \end{equation*}
  where $P_S\to S$ is the projective cover. 
\end{proof}

As far as recursion goes, we now have 

\begin{corollary}\label{cor.rec}
  Let $G$ be a finite group as before, and $S,M\in \mathrm{mod}~kG$ a simple and an arbitrary $G$-module respectively. For exact functors $F:\mathrm{mod}~kG\to \mathrm{Vect}^f$ as in \Cref{th.f-per} or $F=\mathrm{Hom}_G(S,-)$ the sequence
  \begin{equation*}
    n\mapsto \dim F(\Omega^nM) 
  \end{equation*}
  is eventually linearly recursive, and the same goes for $\Omega^{-n}$. 
\end{corollary}
\begin{proof}
  This follows from \Cref{pr.lsoc-poly,th.f-per} and the fact that eventually quasipolynomial sequences are eventually linearly recursive.
\end{proof}

Next, note that for every Laurent polynomial $P\in \bN[x^{\pm 1}]$ we can talk about the functor $P(\Omega)$ (written $P\Omega$ for brevity), with addition being interpreted as direct sum. We have the following amplification of \Cref{cor.rec}.

\begin{theorem}\label{th:po}
  For $F$ and $M$ as in \Cref{cor.rec} and an eventually linearly recursive sequence of polynomials
  \begin{equation*}
    \cP=(P_n)_n\subset \bN[x]
  \end{equation*}
  the sequences
  \begin{equation*}
    n\mapsto \dim F(P_n\Omega M)
  \end{equation*}
  and
  \begin{equation*}
    n\mapsto \dim F(P_n\Omega^{-1}M) 
  \end{equation*}
  are eventually linearly recursive. 
\end{theorem}
\begin{proof}
  To fix ideas, we prove the version about $\Omega$. Denoting
  \begin{equation*}
    a_n=\dim F(\Omega^n M)\text{ and } b_n=\dim F(P_n\Omega M)
  \end{equation*}
  we have
  \begin{equation*}
    {\bf b}:=(b_n)_n = \cP\triangleright {\bf a} \text{ for } {\bf a}:=(a_n)_n
  \end{equation*}
  with `$\triangleright$' as in \Cref{def:op}. The conclusion thus follows from \Cref{pr:qpn}.
\end{proof}

On the other hand, for {\it Laurent} (as supposed to ordinary) polynomials we have the following version.

\begin{theorem}\label{th:lrnt}
  For $F$ and $M$ as in \Cref{cor.rec} and an eventually linearly recursive sequence of Laurent polynomials
  \begin{equation*}
    \cP=(P_n)_n\subset \bN[x^{\pm 1}]
  \end{equation*}
  the sequence
  \begin{equation}\label{eq:pno}
    n\mapsto \dim F(P_n\Omega M)
  \end{equation}
  is algebraic.
\end{theorem}
\begin{proof}
  We can proceed as in the proof of \Cref{th:po}, this time using \Cref{th:pab} and noting that the sequence \Cref{eq:pno} is (essentially, up to irrelevant shifts) $\cP\triangleright({\bf a},{\bf b})$ for
  \begin{equation*}
    {\bf a}=(\dim F(\Omega^n M))_n\quad\text{and}\quad {\bf b}=(\dim F(\Omega^{-n} M))_n. 
  \end{equation*}
\end{proof}


\section{Invariant sequences for Omega-algebraic modules}\label{se.omega}

Recall the definitions of $\Omega$ and $\Omega^{\pm}$-algebraic modules from the Introduction (\Cref{def:omega} and \Cref{def:omega+}). We introduce some notation:
\begin{itemize}
\item Let $M$ be a finite-dimensional $\Omega$-algebraic $kG$-module.
\item Let $N_1,...,N_s$ be the $\Omega$-orbit representatives of the various non-projective indecomposable summands that appear in $M^{\otimes n}$, including $N_1=k$ for $k=M^{\otimes 0}$. 
\item Let $T=(t_{ij})$ be the $k \times k$ matrix whose rows give the effect of tensoring with $M$. So,
  \begin{equation}\label{eq:tmat}
    core_G(M \otimes N_i)=\bigoplus_{j=1}^s t_{ij}(N_j)
  \end{equation}
where $t_{ij}$'s are Laurent polynomials in $\Omega$.
\end{itemize}

\begin{proposition} \label{th.indec}
  If $M$ is an {\it Omega-algebraic} non-projective indecomposable $G$-module, then the sequence $c_n^G(M)$ is the sum of the dimensions of the entries of the first row of the matrix $T^n$, i.e.
  \begin{equation*}
    c_n^G(M)=\sum_{j=1}^s \operatorname{dim} t^{(n)}_{1j}(N_j)
  \end{equation*}
where $T^n=(t^{(n)}_{ij})\in M_s(\bN[\Omega^{\pm 1}])$. 
\end{proposition}

\begin{proof}
Letting $N_1=k$ as mentioned before, we have 
\begin{eqnarray*}
  core_G(M)&=&\bigoplus_{j=1}^s t_{1j}(N_j) \\
  core_G(M^{\otimes 2})&=& \bigoplus_{j=1}^s core_G(t_{1j}(M \otimes N_j))\\
                     &=& \bigoplus_{j=1}^s t_{1j} \Big (\bigoplus_{l=1}^s t_{jl} (N_l) \Big ),
\end{eqnarray*}
etc. The proof follows by induction.
\end{proof}

\begin{remark}\label{re:stbl}
  \Cref{th.indec} hinges on the fact that when regarded as functors on the {\it stable module category} of $G$ (e.g. \cite[\S 2.1]{bnsn-rep1}) the functors $M\otimes -$ and $\Omega^{\pm 1}$ commute. 
\end{remark}

\begin{corollary}
  Let $M$ be an {\it Omega-algebraic} $G$-module. Let $N_1=k,\cdots ,N_s$ be the $\Omega$-orbit representatives of the various non-projective indecomposable summands that appear in $M^{\otimes n}$ with $N_1,...,N_r$ being the $\Omega$-orbit representatives of the non-projective indecomposable summands of $M$, with \begin{equation*}core_G(M)=\bigoplus_{i=1}^r q_{i}(N_i)
\end{equation*} where $q_{i}$'s are Laurent monomials in $\Omega$.  Let $T$ be the matrix that gives the effect of tensoring with $M$. Then we have,
\begin{equation*}c_n^G(M)=\sum_{i=1}^r \sum_{j=1}^s \operatorname{dim} q_ip_{ij}(N_j)
\end{equation*}
where $T^{n}=(p_{ij})$ with $p_{ij}$'s being Laurent polynomials in $\Omega$.
\end{corollary}

\begin{proof}
Let \begin{equation*}core_G(M \otimes N_i)=\bigoplus_{j=1}^s t_{ij}(N_j)
\end{equation*}
where $t_{ij}$'s are Laurent monomials in $\Omega$.
\begin{eqnarray*}
core_G(M \otimes M)&=& core_G(M \otimes \bigoplus_{i=1}^r q_{i}(N_i))\\
&=&core_G(\bigoplus_{i=1}^r q_{i}(M \otimes N_i))\\
&=& \bigoplus_{i=1}^r q_{i} \Big (\bigoplus_{j=1}^s t_{ij} (N_j) \Big )
\end{eqnarray*}
The proof now follows from \Cref{th.indec}.
\end{proof}

\begin{theorem}\label{th:cj142-weak}
  \cite[Conjecture 14.2]{bs} holds for Omega$^+$ and Omega$^-$-algebraic modules $M$. 
\end{theorem}
\begin{proof}
  The two claims are analogous, so we focus on the Omega$^+$ case.
  
  By \Cref{th.indec} and the assumption that $M$ is Omega$^+$-algebraic $c_n^G(M)$ is the sum of the dimensions of
  \begin{equation*}
    t_{1j}^{(n)} N_j,\ 1\le j\le s,
  \end{equation*}
  where $t_{1j}^{(n)}$ are the respective entries of the $n^{th}$ power $T^n$ of an $s\times s$ matrix over $\bN[\Omega]$.

  By the Cayley-Hamilton theorem (over the ring of polynomials in $\Omega$) the sequences $(t_{1j}^{(n)})_n$ are all recursive. The conclusion now follows from \Cref{th:po} applied to said sequences (with $N_j$ respectively in place of $M$).
\end{proof}

As an immediate consequence, we have

\begin{corollary}\label{cor:sufflrg}
  For any Omega-algebraic $M\in \mathrm{mod}~kG$ the sequences
  \begin{equation*}
    (c_n^G(\Omega^d M))_n\text{ and }(c_n^G(\Omega^{-d} M))_n
  \end{equation*}
  are eventually linearly recurrent for sufficiently large $d$. 
\end{corollary}
\begin{proof}
  Indeed, for sufficiently large $d$ $\Omega^d M$ is Omega$^+$-algebraic while $\Omega^{-d} M$ is Omega$^-$-algebraic. The conclusion follows from \Cref{th:cj142-weak}.
\end{proof}

As to the $\Omega$-algebraic analogue of \Cref{th:cj142-weak}:

\begin{theorem}\label{th:wkagain}
  For an $\Omega$-algebraic $M$ the sequence $(c_n^G(M))_n$ is algebraic.
\end{theorem}
\begin{proof}
  As in the proof of \Cref{th:cj142-weak}, except now the polynomials are Laurent and we use \Cref{th:lrnt} in place of \Cref{th:po}.
\end{proof}

On the other hand, the number $s_n^G(M)$ of indecomposable summands of $core_G(M^{\otimes n})$ is better behaved:

\begin{theorem}\label{th:snbetter}
  For an $\Omega$-algebraic $M$ the sequence $(s_n^G(M))_n$ is eventually linearly recursive.
\end{theorem}
\begin{proof}
  Since $\Omega$ preserves (in)decomposability, the effect on $s_n$ of tensoring by $M$ is given as in \Cref{eq:tmat}, upon substituting $1$ for $\Omega$ in the matrix entries $t_{ij}$ and also $1$ for each indecomposable $N_j$. In other words, $s_n^G(M)$ can be recovered as the sum of the entries of $A^n$ for a scalar matrix $A$; clearly, this is a recursive sequence.
\end{proof}

\begin{remark} It follows from \Cref{th:snbetter} and \cite[Theorem 13.2]{bs} that for an $\Omega$-algebraic module $M$, the invariant $\gamma_G(M)$, as defined in \cite[Definition 1.1]{bs}, will always be an algebraic integer.
\end{remark}

\section{Examples} \label{se.examples}

Recall that $c_n^G(M)$ is the sequence of dimensions of the core of $M^{\otimes n}$ whereas $s_n^G(M)$ is the sequence of the number of indecomposable summands of the core of $M^{\otimes n}$. In this section, we will see some examples to demonstrate that this sequence is eventually polynomial or recurrent.

\begin{enumerate}[(1)]
\item Let $G$ be the cyclic group of order 7, $k$ be a field of characteristic 7 and $M$ be the indecomposable $kG$-module of dimension 2.\\
Then, the sequence  \begin{equation*}c_n^G(M)=<2,4,8,16,32,57,114,193,386,639,1278,2094,6829,...>,
\end{equation*}
and
\begin{equation*}s_n^G(M)=<1,2,3,6,10,19,33,61,108,197,352,638,1145,2069,...>.
\end{equation*}
The sequences above satisfy the relation \begin{equation*}x_n=5x_{n-2}-6x_{n-4}+x_{n-6}.
\end{equation*}

\item Let $G=\mathcal{S}_{10}$, $k$ be a field of characteristic 5 and $M$ be the permutation module of the symmetric group $\mathcal{S}_{10}$ labelled by the partition $\lambda=(9,1)$.\\
Then, the sequence
\begin{equation*}s_n^G(M)=<1,4,19,94,469,2344,...>,
\end{equation*}
which satisfies the relation \begin{equation*}x_n=x_{n-1}+25x_{n-2}-25x_{n-3}.
\end{equation*}

\item Let $G=\mathcal{S}_9$, $k$ be a field of characteristic 3. Young modules are indecomposable summands of the permutation modules of the symmetric group and are also labelled by partitions of $n$ as shown in \cite{KErdmannYoung}. Let $M=Y^\lambda$, the Young module corresponding to the partition $\lambda=(7,2)$.\\
Then, the sequence
\begin{equation*}s_n^G(M)=<1,4,35,310,2789,25096,...>,
\end{equation*}
which satisfies the relation \begin{equation*}x_n=9x_{n-1}+x_{n-2}-9x_{n-3}.
\end{equation*}


\item
  Let $G=\langle g, h \rangle \cong \bZ/3 \times \bZ/3$ and $k=\bF_3$. Let $M$ be the six-dimensional module given by the following matrices:
\begin{equation*}g \mapsto \begin{pmatrix}
1 & 0 & 1 & 0 & 0 & 0 \cr
0 & 1 & 0 & 0 & 0 & 0 \cr
0 & 0 & 1 & 0 & 1 & 0 \cr
0 & 0 & 0 & 1 & 0 & 1 \cr
0 & 0 & 0 & 0 & 1 & 0 \cr
0 & 0 & 0 & 0 & 0 & 1
\end{pmatrix}
\hspace{1cm}
h \mapsto \begin{pmatrix}
1 & 0 & 0 & 1 & 0 & 0 \cr
0 & 1 & 0 & 0 & 1 & 0 \cr
0 & 0 & 1 & 0 & 0 & 1 \cr
0 & 0 & 0 & 1 & 0 & 0 \cr
0 & 0 & 0 & 0 & 1 & 0 \cr
0 & 0 & 0 & 0 & 0 & 1
\end{pmatrix}
\end{equation*}
This is precisely \cite[Example 15.1]{bs}, on which we now elaborate. First, as noted in loc.cit., $M$ is Omega-algebraic: if $N=k_{\langle g\rangle}\uparrow^G$, then
\begin{align*}
  core_G(M \otimes M) &\cong \Omega(M) \oplus \Omega^{-1}(M^*) \oplus N\\
  core_G(M \otimes M^*) &\cong \Omega^{-1}(M) \oplus \Omega(M^*) \oplus N\numberthis\label{eq:motimes}\\
  core_G(M \otimes N) &\cong 3\Omega(N). 
\end{align*}


\begin{remark}
  It is also mentioned in \cite[Example 15.1]{bs} that $M$ is not algebraic. Indeed, it can be shown that in Craven's taxonomy of 6-dimensional indecomposable $G$-modules, it belongs to class P in the table from \cite[\S 3.3.5]{cr-phd}. Indeed, $M$
  \begin{itemize}
  \item has socle layers of dimensions 2,2,2, as can easily be seen either directly or from the diagram displayed next to the two matrices in \cite[Example 15.1]{bs};    
  \item has dual with socle layers of dimensions 2,3,1, as is again easily seen from the fact that in passing from $M$ to $M^*$ one can simply transpose the matrices corresponding to the generators $g$ and $h$.
  \end{itemize}
  Jointly, these remarks eliminate all possibilities in \cite[table, \S 3.3.5]{cr-phd} except for classes P and I$^*$. The only distinction noted in loc.cit. between the two is the cardinality of the set of conjugates under the action of the automorphism group $\mathrm{Aut}~G$: 4 for P and 8 for I$^*$. Now, $\mathrm{Aut}~G$ has order 48, so it will be enough to check whether the isotropy group of (the isomorphism class of) $M$ contains a subgroup of order 4: if it does the class must be P, and it will be I$^*$ otherwise.

  To conclude, simply note that the Klein 4-group generated by the automorphisms that square one of the two generators and fix the other one fixes $M$: conjugation by $\mathrm{diag}(2,1,2,1,2,1)$ maps  
  \begin{equation*}
    g\mapsto g\quad\text{and}\quad h\mapsto h^2,
  \end{equation*}
  whereas conjugation by
  \begin{equation*}
    \begin{pmatrix}
      1&0&0&0&0&0\\
      0&1&0&0&0&0\\
      0&0&2&0&1&0\\
      0&0&0&1&0&0\\
      0&0&0&0&1&0\\
      0&0&0&0&0&2
    \end{pmatrix}
  \end{equation*}
  gives the other automorphism
  \begin{equation*}
    g\mapsto g^2\quad\text{and}\quad h\mapsto h.
  \end{equation*}
  To reiterate, this means that the isotropy group of $M$ in $\mathrm{Aut}~G$ has order divisible by 4, and hence the size of the orbit must divide $\frac{48}{4}=12$. In particular that size cannot be 8, ruling out class I$^*$ from \cite[table, \S 3.3.5]{cr-phd}.
\end{remark}

Let $T$ be the $3 \times 3$ matrix whose rows give the effect of tensoring the non-projectives with $M$.
\begin{table}[ht]
\centering 
\begin{tabular}{c | c c c} 
$ $ & $M$  & $M^*$ & $N$  \\ [0.5ex] 
\hline 
$M$ & $\Omega$  & $\Omega^{-1}$ & $1$ \\ 
$M^*$ & $\Omega^{-1}$ & $\Omega$ & $1$ \\
$N$ & $0$ &  $0$ & $3\Omega$ \\ [1ex] 
\hline 
\end{tabular}
\end{table} \\
\smallskip
Hence,
\begin{equation*}T= \begin{pmatrix}
\Omega & \Omega^{-1} & 1 \cr
\Omega^{-1} & \Omega & 1 \cr
0 & 0 & 3\Omega 
\end{pmatrix}
\hspace{1cm} \text{and} \hspace{1cm}
T^n= \begin{pmatrix}
A_n & B_n & C_n \cr
B_n & A_n & C_n \cr
0 & 0 & (3\Omega )^n
\end{pmatrix}
\end{equation*}
where $A_n,B_n$ and $C_n$ are Laurent polynomials in $\Omega$ described as follows:\\
\begin{align*}
A_n &= \sum_{i=0}^{\lfloor n/2 \rfloor}{n \choose 2i} \Omega ^{(n-4i)}\numberthis \label{eq:an}\\
B_n &= \sum_{i=0}^{\lfloor n/2 \rfloor}{n \choose 2i+1} \Omega ^{(n-(4i+2))}\numberthis \label{eq:bn}\\
C_n &= \sum_{k=1}^{n}\alpha_n ^{(k)} \Omega ^{(n-(2k-1))}
\end{align*}

where
\begin{equation*}\alpha_n ^{(k)} = \sum_{i=0}^{k}(-1)^i \Bigg [{k+1 \choose i+1} + 2 {k \choose i}\Bigg ] \alpha_{n-1-i}^{(k)}
\end{equation*}
are linear recurrence relations with the initial conditions:
\begin{eqnarray*}
\alpha_t^{(k)} &=& 0, \hspace{.5cm} \text{if } \hspace{.5cm} 0\leq t<k\\
\alpha_k^{(k)} &=& 1.
\end{eqnarray*}

The characteristic equation of $T$ is 
\begin{equation*}
  x^3-5\Omega x^2-(\Omega^{-2}-7\Omega^2)x-(3\Omega^3-3\Omega^{-1})=0.
\end{equation*}
So by the Cayley-Hamilton Theorem over the ring $\mathbb{C}[\Omega, \Omega^{-1}]$, the sequences $A_n$, $B_n$ and $C_n$ satisfy the recurrence relation
\begin{equation*}
x_n=5\Omega x_{n-1}+(\Omega^{-2}-7\Omega^2)x_{n-2}+(3\Omega^3-3\Omega^{-1})x_{n-3}
\end{equation*}
for $n\geq4$.

The number $c_n^G(M)$ can now be recovered as
\begin{equation}\label{eq:anbncn}
  c_n^G(M) = \dim A_n(M) + \dim B_n (M^*) + \dim C_n (N).
\end{equation}
We do not know whether this ends up being eventually linearly recursive (as opposed to just algebraic \Cref{th:wkagain}), but we end with a few remarks on the matter.

First, note that the third summand $\dim C_n(N)$ is unproblematic here, as it is indeed linearly recursive. To see this, note that $N$ is periodic because its restriction to the maximal subgroup $\langle h\rangle\subset G$ is projective \cite[Corollary 2.24.7]{ben-mod}. Furthermore, this implies that it is periodic of period 1 or 2 \cite[Theorem 6.3]{car-str}. But then the recursion
\begin{equation*}
  C_{n+1} = (\Omega+\Omega^{-1})C_n + (3\Omega)^n
\end{equation*}
implies that we can substitute $\Omega$ for $\Omega^{-1}$ in the formula above, and can hence conclude as in the $\Omega^{+}$-algebraic case covered by \Cref{th:cj142-weak}.

The other two terms in \Cref{eq:anbncn} seem more difficult to tackle. Observe that since $M$ is not periodic, it must have complexity 2 (because we are working over a 3-group of rank 2 \cite[Theorem 2.24.4 (xv)]{ben-mod}). This implies that $\dim \Omega^n M$ and all of its analogues ($\dim \Omega^n(M^*)$, etc.) are eventually polynomials of degree 1.

It follows from the above, for instance, that the multiplicity of $M$ in $core_G(M^{\otimes n})$ cannot be eventually linearly recursive: the number of terms in \Cref{eq:an,eq:bn} that can be isomorphic to $M$ is, for dimension reasons, uniformly bounded in $n$ and concentrated around the middle of the range in either of those two sums, so the multiplicities in question are sums of binomial coefficients of the form
\begin{equation*}
  \binom{n}{\lfloor \frac n2\rfloor + k}
\end{equation*}
with $k$ ranging over a fixed interval centered at 0. Such binomial coefficients do not form linearly recursive sequences: see e.g. \cite[Example 6.3.2]{enum2}.

On the other hand, as per \Cref{th:snbetter}, $s_n:=s_n^G(M)$ is a recursive sequence: \Cref{eq:motimes}, together with the fact that $M\otimes-$ and $\Omega^{\pm 1}$ commute module projective summands, makes it clear that each iteration of tensoring with $M$ will triple the number of indecomposable, non-projective summands. We thus have $s_n=3^{n-1}$.



\end{enumerate}



\addcontentsline{toc}{section}{References}

\Addresses

\end{document}